\documentclass
[
11pt]
{amsart}%

\usepackage{hyperref}
\hypersetup{
    colorlinks=true,
    linkcolor=blue,
    filecolor=magenta,
    urlcolor=cyan,
    citecolor=blue
}
\usepackage{cite}
\usepackage{amssymb}
\usepackage{amsmath}
\usepackage{amsthm}
\usepackage{amsfonts}
\usepackage{graphicx}
\usepackage{bbm}
\usepackage[usenames,dvipsnames]{xcolor}
\usepackage[toc,page]{appendix}
\usepackage{subfigure}
\usepackage{mathtools}
\usepackage{bm}
\usepackage[normalem]{ulem}
\usepackage{comment}
\usepackage{MnSymbol}
\usepackage{mathrsfs}

\usepackage{tikz-cd}
\newcommand{\beq}{\begin{equation}}
\newcommand{\eeq}{\end{equation}}

\newcommand{\bea}{\begin{eqnarray}}
\newcommand{\eea}{\end{eqnarray}}
\newcommand{\beas}{\begin{eqnarray*}}
\newcommand{\eeas}{\end{eqnarray*}}

\newtheorem{theorem}{Theorem}[section]

\newtheorem{proposition}[theorem]{Proposition}
\newtheorem{prop}[theorem]{Proposition}
\newtheorem{corollary}[theorem]{Corollary}
\newtheorem{lemma}[theorem]{Lemma}

\theoremstyle{definition}
\newtheorem{definition}[theorem]{Definition}
\newtheorem{example}[theorem]{Example}
\newtheorem{assumption}[theorem]{Assumption}
\newtheorem{remark}[theorem]{Remark}
\newtheorem{examples}[theorem]{Examples}

\newtheorem*{acknowledgement}{Acknowledgement}

\newcommand{\R}{\mathbb R}

\newcommand{\cS}{\mathcal S}

\newcommand{\tth}{\tilde{\theta}}
\newcommand{\x}{\mathbf{x}}
\newcommand{\y}{\mathbf{y}}
\newcommand{\z}{\mathbf{z}}
\newcommand{\B}{\mathbf{B}}

\newcommand{\tmu}{\tilde{\mu}}
\newcommand{\tce}{\tilde{\mathcal{E}}}

\newcommand{\tL}{\tilde{\mathcal{L}}}
\newcommand{\tlam}{{\lambda}'}

\newcommand{\tq}{\tilde{q}}
\newcommand{\tp}{\tilde{p}}
\newcommand{\tP}{\tilde{P}}
\newcommand{\tE}{\tilde{E}}

\begin{document}

\title[Rigidity of the subelliptic heat kernel]{Rigidity of the subelliptic heat kernel on $\operatorname{SU}\left( 2 \right)$}
\author[Gordina]{Maria Gordina{$^{\dagger}$}}
\thanks{\footnotemark {$\dagger$} This research was supported in part by NSF Grant DMS-1954264.}
\address{Department of Mathematics\\
University of Connecticut\\
Storrs, CT 06269, USA} \email{maria.gordina@uconn.edu}
\author[Wang]{Jing Wang{$^{\ddag}$}}
\thanks{\footnotemark {$\ddag$} Research was supported in part by NSF Grant DMS-2246817.}
\address{Department of Mathematics\\
Purdue University\\
West Lafayette, IN 47907, USA} \email{jingwang@purdue.edu}



\begin{abstract}
We study heat kernel rigidity for the Lie group $\operatorname{SU}\left( 2 \right)$ kernel equipped with a sub-Riemannian structure. We prove that a metric measure space equipped with a heat kernel of a special form is bundle-isometric to the Hopf fibration $\operatorname{U}\left( 1 \right)\to \operatorname{SU}\left( 2 \right)\to \mathbb{CP}^1$, which coincides with the sub-Riemannian sphere $\operatorname{SU}\left( 2 \right)$.
\end{abstract}

\maketitle

\tableofcontents

\section{Introduction}

Rigidity and stability problems have been studied both in analysis and geometry in different settings. One of the motivations is the fundamental question of how one can characterize the underlying geometry based on the observed transition density of a diffusion process on the space. This goes back to the classical question \emph{Can one hear the shape of a drum?} by M.~Kac in \cite{Kac1966}. Such questions have practical applications in shape matching \cite{BronsteinABronsteinMimmelMahmoudiSapiro} and segmentation \cite{De_GoesGoldensteinVelho2008}, as heat kernels are invariant under deformations in terms of isometric maps and preserve all of the intrinsic information about the shape. Applied analysts have developed various tools such as the heat kernel signature \cite{SunOvsjanikovGuibas2009}, diffusion maps \cite{CoifmanLafon2006}, eigenfunctions embedding \cite{BelkinNiyogi2001}. 

There have been many classical and new results addressing such rigidity problems in the Riemannian setting connecting geometry and properties of heat kernels, including Colding's rigidity \cite{colding1997} with a new proof given in the recent paper by Carron and Tewodrose \cite{CarronTewodrose2022}.  However, to the authors' knowledge, no results are available addressing heat kernel rigidity in the sub-Riemannian setting. A sub-Riemannian manifold is a smooth  manifold $M$ equipped with a fiberwise inner product on a sub-bundle $\mathcal{H}$ of its tangent bundle which is bracket-generating. There is a genuine distance $d_{cc}$ on $M$ given by the smallest length of all absolutely continuous curves that are tangential to $\mathcal{H}$. This distance is called a sub-Riemannian or Carnot-Carath\'eodory distance. Owing to its inherent degeneracy, sub-Riemannian geometry provides a rich body of structures, while notably its geometric aspects are considerably rougher than Riemannian manifolds. At the same time, it is a natural setting for the geometric control theory. The aim of our paper is to prove a heat kernel rigidity result for  the $3$-dimensional compact model space $S^3$ equipped with a sub-Riemannian structure, also known as the Hopf fibration, from a given transition density of a diffusion on it.  

It is often more convenient to describe the sub-Riemannian structure on $S^3$ using its identification with the Lie group $\operatorname{SU}(2)$, that is, the group of $2\times 2$ complex unitary matrices of the determinant $1$. A basis of its Lie algebra $\mathfrak{su}(2)$ is given by
 \[
\text{ }X=\left(
\begin{array}{cc}
~0~ & ~1~ \\
-1~& ~0~
\end{array}
\right) ,\text{ }Y=\left(
\begin{array}{cc}
~0~ & ~i~ \\
~i~ & ~0~
\end{array}
\right),
Z=\left(
\begin{array}{cc}
~i~ & ~0~ \\
~0~ & -i~
\end{array}
\right),
\]
which satisfy the bracket generating relations
\[
[Z, X]=2Y,\quad [X,Y]=2Z,\quad [Y,Z]=2X.
\]
We use the same notation $X, Y, Z$ for the corresponding left-invariant vector fields on $\operatorname{SU}(2)$. The sub-Riemannian structure is then given by the sub-bundle $\mathcal{H}(\operatorname{SU}(2))=\{\mathcal{H}_x, x\in \operatorname{SU}(2) \}$, where $\mathcal{H}_x=\operatorname{Span}\{X(x), Y(x)\}$, and a metric on $\mathcal{H}(\operatorname{SU}(2))$ is such that $X, Y$ are orthonormal. Then the operator 
\[
L:=X^2+Y^2
\]
is hypoelliptic by \cite{Hormander1967a}, often called a sub-Laplacian on $\operatorname{SU}(2)$. Let $P_t:=e^{tL}$ be the heat semigroup generated by $L$, and by $p_t$ we denote its (smooth) heat kernel. The explicit expression for $p_t$ has been obtained by Bonnefont-Baudoin in \cite{BaudoinBonnefont2009}. Using the cylindrical coordinates
 
\[
(r,\theta, z) \longrightarrow \exp \left(r \cos \theta X +r \sin \theta Y \right) \exp (z Z)
\] 
they show that the sub-elliptic heat kernel on $\operatorname{SU}\left( 2 \right)$ for the semigroup starting from the identity is given by
\begin{equation}\label{eq-kernel-1}
p_t(r,\theta)=\sum_{n=-\infty}^{\infty} \sum_{k=0}^\infty(2k+|n|+1)e^{-(4k(k+|n|+1)+2|n|)t}e^{in\theta}(\cos r)^{|n|}P_k^{0, |n|}(\cos 2r)
\end{equation}
for $t>0$, $0\leqslant r < \frac{\pi}{2}$ and $\theta \in [-\pi,\pi]$, where
\[
P_k^{0,|n|}(x)=\frac{(-1)^k}{2^kk!(1+x)^{|n|}}\frac{d^k}{dx^k}\left( (1+x)^{|n|}(1-x^2)^k \right)
\]
is the Jacobi polynomial.

In this paper, our aim is to show that if a Dirichlet metric measure space admits a heat kernel taking the form of \eqref{eq-kernel-1}, then this space can be identified with the sphere $S^3$ equipped with a  sub-Riemannian structure. There are several challenges compared to the Riemannian setting. First, unlike in \cite{CarronTewodrose2022} a sub-Riemannian distance, such as the Carnot-Carath\'eodory distance $d_{cc}$, does not appear in the heat kernel expression \eqref{eq-kernel-1}. Instead we need to work with two semi-metric $r$ and $\theta$. We should note that the connection between $d_{cc}$ and $(r,\theta)$ has been established in \cite{BaudoinBonnefont2009}, however this connection is not explicit. Another obvious obstacle is that to identify a sub-Riemannian structure we need to show the existence of a bracket-generating sub-bundle of the tangent bundle. However, starting with a general metric measure space, we do not have a smooth structure a priori.  Unlike Riemannian manifolds, we do not expect sub-Riemannian manifolds to be embeddible as metric spaces into Euclidean spaces. 
For example, Semmes in \cite{Semmes1996a} used the Heisenberg group as an example of a doubling metric space which is not bi-Lipschitz equivalent to a subset of any Euclidean space. Therefore finding isometry maps directly is not possible. 
The main idea we exploit is to take advantage on the connection between $\operatorname{SU}(2)$ equipped with a sub-Riemannian structure and the Hopf fibration
\[
S^1 \longrightarrow  \operatorname{SU}(2) \longrightarrow \mathbb{CP}^1,
\] 
where $S^1$ is the unit circle and $\mathbb{CP}^1$ is the complex projective space that can be identified with a $2$-dimensional sphere of radius $\frac12$. Escobales in \cite{Escobales1975} characterized the Hopf fibration $S^1 \longrightarrow \operatorname{SU}(2) \longrightarrow \mathbb{C}P^1$ as the unique Riemannian submersion of $S^3$ with totally geodesic fibers.  It is known that such a Riemannian submersion induces the sub-Riemannian structure on $\operatorname{SU}(2)$.

Our strategy is to start with a space $M$ with semi-metrics $r$ and $\theta$, and a Borel measure $\mu$. We first use $r$ and $\theta$ to induce a metric $\delta$. Then by defining the equivalence relation $x\sim y$ as $r(x,y)=0$, we obtain the  quotient space denoted as $\B:=M/\sim$. 

We then reduce the problem to finding bundle isometries $(\Phi, \Psi)$ from $\pi$ to $\Pi$ such that the following diagram below commutes. 

\begin{equation}
\begin{tikzcd}\label{eq-diagram}
  (M,\delta) \arrow[r, "\Phi"] \arrow[d, "\pi"]
    &  (\cS^3, d_{\cS^3})  \arrow[d, "\Pi"] \\
   (\B,r) \arrow[r, "\Psi"]
&  (\cS^2,d_{\cS^2}) 
\end{tikzcd}
\end{equation}
Here $(\cS^3, d_{\cS^3})$ and $(\cS^2,d_{\cS^2})$ are essentially the $\operatorname{SU}(2)$ and $\mathbb{C}P^1$ equipped with the standard Riemannian metrics, and $\Pi$ is a Riemannian submersion from $(\cS^3,d_{\cS^3})$ to $(\cS^2,d_{\cS^2})$ with totally geodesic fibers.

The rest of this paper is organized as follows. In Section~\ref{sec-setting} we introduce the setting and preliminaries for the metric measure spaces. In Section~\ref{sec-isom-1} we define the first isometry $\Phi:  (M,\delta)\to (\cS^3, d_{\cS^3})$. In Section~\ref{sec-isom-2} we introduce the second isometry $\Psi: (\B,r) \to  (\cS^2,d_{\cS^2})$. Lastly in Section~\ref{sec-Riem-sub} we show that $\Pi: (\cS^3, d_{\cS^3})\to (\cS^2, d_{\cS^2})$ is a sub-Riemannian submersion with totally geodesic fibers.

\section{Preliminaries}\label{sec-setting}

\subsection{Dirichlet forms, semigroups and heat kernels}\label{sec-Diri-M} 
Suppose a metric space $\left( X, d \right)$ is equipped with a non-negative regular $\sigma$-finite Borel measure $\mu$ on the Borel measurable space $\left( X, \mathcal{B}\left( X \right) \right)$. We call the triple $\left( X, d, \mu \right)$ a \emph{metric measure space}.  
We describe briefly the one-to-one correspondence between Markov semigroups, Dirichlet forms and heat kernels on $L^{2}\left( X, \mu \right)$. We refer to monographs \cite{ChenFukushimaBook2012, FukushimaOshimaTakedaBook2011} for more details, though in our exposition we rely on \cite{Grigor'yan2010a}.

First recall how (abstract) heat kernels can be defined on a metric measure space, see e.~g   \cite[Section~2.1]{Grigor'yan2010a}.

\begin{definition}[Abstract heat kernels]\label{d.HeatKernel} We call a family $\left\{ p_{t}\left( x, y \right) \right\}_{t>0}$ of measurable functions $p_{t}\left( x, y \right)$ on $X \times X$ 
a \emph{heat kernel} if the following conditions are satisfied, for $\mu$-almost all $x, y \in  M$ and all $s, t > 0$.
\begin{itemize}
\item[(i)]\label{i.positive} Positivity $p_{t}\left( x, y \right)  \geqslant 0$;

\item[(ii)]\label{i.symmetry} Symmetry $p_{t}\left( x, y \right)=p_{t}\left( y, x \right)$;

\item[(iii)]\label{i.conservative} Stochastic completeness (conservative heat kernel): for $\mu$-almost all $x \in  X$  and all $t > 0$ 
\[
\int_{X} p_{t}\left( x, y \right) d\mu(y)= 1;
\]
\item[(iv)]\label{i.semigroup} The semigroup property (Chapman-Kolmogorov inequalities)
\[
p_{s+t}\left( x, y \right)=\int_{X} p_{s}\left( x, z \right)p_{t}\left( z, y \right)d\mu(z);
\]
\item[(v)]\label{i.approximation} Approximation of identity:  for all $f \in L^{2}\left( X, \mu \right)$
\[
\int_{X} f\left( y \right) p_{s}\left( x, y \right)d\mu(y)\xrightarrow[{t\to 0+}]{L^2} f(x). 
\]
\end{itemize}
\end{definition}

There is one-to one correspondence between such heat kernels on metric measure spaces and heat (convolution) semigroup which are strongly continuous symmetric Markovian semigroups on $L^{2}\left( X, \mu \right)$. Namely, any heat kernel defines a heat semigroup $\left\{ P_{t} \right\}_{t \geqslant 0}$ where $P_{0}=\operatorname{Id}$ and $P_{t}, t >0$ is an operator in on $L^{2}\left( X, \mu \right)$ defined by
\[
P_{t}f\left( x \right) = \int_{X} f\left( y \right) p_{t}\left( x, y \right)d\mu(y),\quad  f \in L^{2}\left( X, \mu \right).
\]
Then by properties of the heat kernel in Definition~\ref{d.HeatKernel} we see that $P_{t}$ is Markovian semigroup, it is a bounded operator on $L^{2}\left( X, \mu \right)$. In addition, it is a contraction operator which is symmetric and, therefore, self-adjoint. Thus $\left\{P_{t} \right\}_{t \geqslant 0}$ is a strongly continuous symmetric Markovian semigroup on $L^{2}\left( X, \mu \right)$. 

For a strongly continuous symmetric Markovian semigroup $P_{t}$ on $L^{2}\left( X, \mu \right)$, we can define the \emph{infinitesimal generator} $L$ of the semigroup $P_{t}$ 
\[
Lf:= \lim\limits_{t \to 0^{+}}\frac{P_{t}f-f}{t},
\]
where the limit is taken in the $L^{2}$-norm. The domain $\mathcal{D}_{L}$ of the generator
$L$ are all functions $f \in L^{2}\left( X, \mu \right)$ for which the limit above exists. Then by the Hille-Yosida theorem, the operator $L$ is densely defined. Moreover, $L$ is a self-adjoint negative-definite operator, since the semigroup $P_{t}$ is self-adjoint
and Markovian. 

Let $\mathcal{E}$ be the associated Dirichlet form defined by 
\[
\mathcal{E}(f,g):= \lim\limits_{t \to 0^{+}}\left(\int_X \frac{P_{t}f-f}{t} g(x)d\mu(x)\right),
\] 
for all $f, g \in  L^2\left( X, \mu \right)$ for which such limit exists. Define

\[
\mathcal{D}_{\mathcal{E}}:=\left\{ f \in L^{2}\left( X, \mu \right): \mathcal{E}(f,f) < \infty\right\},
\]
then $\left( \mathcal{E}, \mathcal{D}_{\mathcal{E}}\right)$ is a Dirichlet form, and $\mathcal{D}_{\mathcal{E}}=\mathcal{D}_{L}$.

Recall that a \emph{Dirichlet form} $\mathcal{E}$ on $\left( X, d \right)$ is a positive-definite symmetric bilinear densely defined closed form $\mathcal{E}: \mathcal{D}_{\mathcal{E}}  \times \mathcal{D}_{\mathcal{E}}  \longrightarrow \mathbb{R}$, where $\mathcal{D}_{\mathcal{E}}$ is a dense subset of $L^{2}\left( X, \mu \right)$, and $\mathcal{E}$ is a closed Markovian symmetric form. This means that $\mathcal{D}_{\mathcal{E}}$ is a Hilbert space with respect to the inner product

\[
\langle f, g \rangle:=\int_{X} fg d\mu+\mathcal{E}\left( f, g \right), f, g \in \mathcal{D}_{\mathcal{E}},
\]
the \emph{Markov property}  of $\mathcal{E}$ means that for any $f \in \mathcal{E}$ we have $\min\left\{ \max\left\{0, f \right\}, 1 \right\} \in \mathcal{E}$ and $\mathcal{E}\left( \min\left\{ \max\left\{0, f \right\}, 1 \right\}, \min\left\{ \max\left\{0, f \right\}, 1 \right\} \right) \leqslant \mathcal{E}\left( f, f \right)$.

A Dirichlet form $\left( \mathcal{E}, \mathcal{D}_{\mathcal{E}}\right)$ is called \emph{regular} if it admits a core, that is, if there exists a subset $\mathcal{C}$ of $C_{0}\left( X \right) \bigcap \mathcal{E}$ that is dense both in $\mathcal{E}$ with respect to the norm induced by the inner product on $\mathcal{E}$, and in $C_{0}\left( X \right)$ with respect to the sup-norm. 

The form $\left( \mathcal{E}, \mathcal{D}_{\mathcal{E}}\right)$ is called \emph{local} if $ \mathcal{E}\left( f, g \right) = 0$ whenever $f, g \in \mathcal{D}_{\mathcal{E}}$ have compact disjoint supports, and it is called \emph{strongly local} if
$ \mathcal{E}\left( f, g \right) = 0$  whenever $f, g \in \mathcal{D}_{\mathcal{E}}$ have compact supports and $f\equiv \text{ const }$ in an open neighborhood of $\operatorname{supp} g$.

\subsection{Assumptions}
In this section we impose necessary assumptions for setting up the problem. Given a set $M$, we would like to equip it with two pseudo-metrics. Let us first recall the definition of a pseudo-metric.

\begin{definition}
A \emph{pseudo-metric} $r$ on $M$ is a non-negative function $r: M\times M \to \R_{\geqslant 0}$ such that for all $x, y, z\in M$

\begin{itemize}
\item[1.] $r(x,x)=0$, 
\item[2.] $r(x,y)=r(y,x)$,
\item[3.] $r(x,y)+r(y,z)\geqslant r(x,z)$.
\end{itemize}
\end{definition}
We then impose the following assumption. 

\begin{assumption}\label{Assump1}
Suppose that there are two pseudo-metrics $r\in[0,\pi/2)$ and $\theta\in[0,\pi)$ on $M$ such that $r(x,y)=\theta(x,y)=0$ implies that $x=y$, and suppose that $M$ is complete with respect to both $r$ and $\theta$.
\end{assumption}

\begin{assumption}[Assumptions on the underlying space]\label{Assump2}
We suppose that $\left( M, r, \theta \right)$ is a topological space equipped with a non-negative regular $\sigma$-finite Borel measure $\mu$ on $\left( M, \mathcal{B}\left( M \right) \right)$. 
\end{assumption}

\begin{assumption}\label{Ass.HeatKernel}(Heat kernel) Suppose  $\left( M, r, \theta, \mu \right)$ satisfies the Assumption \ref{Assump1} and  \ref{Assump2}. We assume that $p_t(x,y)$ is a heat kernel on $M$ taking the form $p_t(x,y)=p(t, r(x,y),\theta(x,y))$,  where
\begin{equation}\label{eq-kernel}
p(t, r,\theta)=\sum_{n=-\infty}^{\infty} \sum_{k=0}^\infty(2k+|n|+1)e^{-(4k(k+|n|+1)+2|n|)t}e^{in\theta}(\cos r)^{|n|}P_k^{0, |n|}(\cos 2r),
\end{equation}
and 
\[
P_k^{0,|n|}(x)=\frac{(-1)^k}{2^kk!(1+x)^{|n|}}\frac{d^k}{dx^k}\left( (1+x)^{|n|}(1-x^2)^k \right)
\]
are Jacobi polynomials.
\end{assumption}

We start by describing properties of the heat kernel $p_{t}\left( x, y \right)$ based on Equation~\eqref{eq-kernel}. Note that from symmetry of pseudo-metrics the heat kernel $p_{t}\left( x, y \right)$ is indeed symmetric. 

\begin{prop}[Properties of the heat kernel]\label{p.HeatKernel}
Suppose  $\left( M, r, \theta, \mu \right)$ satisfies Assumption \ref{Ass.HeatKernel}. Then the measure $\mu$ is a probability measure, and the heat kernel $p_{t}\left( x, y \right)$ is in $L^{2}\left( M \times M, \mu \times \mu \right)$. 
\end{prop}

\begin{proof}
The symmetry of the heat kernel follows from symmetry of pseudo-metrics. 
Now we show that the series in Equation~\eqref{eq-kernel} converges absolutely for any $t>0$. Indeed, Jacobi polynomials satisfy

\[
\max_{x \in [-1, 1]}\left\vert P_k^{0, |n|}(\cos 2r) \right\vert \sim k^{\max\left\{ 0, |n| \right\}}=k^{|n|}, k \to \infty
\]
by \cite[Theorem~7.32.1]{SzegoBook1939} due to Kogbetliantz-S.~Bernstein-Szeg\"{o}.
Therefore

\[
\left\vert p(t, r,\theta) \right\vert\leqslant \sum_{n=-\infty}^{\infty} \sum_{k=0}^\infty(2k+|n|+1)e^{-(4k(k+|n|+1)+2|n|)t}\left\vert P_k^{0, |n|}(\cos 2r) \right\vert \xrightarrow[t \to \infty]{} 1,
\]
and then $p(t, r,\theta) \xrightarrow[t \to \infty]{} 1$. Using the semigroup property (Chapman-Kolmogorov equations), we see that

\begin{align*}
& \int_{M} p_{t/2}\left( x, z \right) p_{t/2}\left( z, y \right) d\mu\left( z \right)=p_{t}\left( x, y \right) \xrightarrow[t \to \infty]{} 1,
\end{align*}
therefore $\mu\left( X \right)=1 < \infty$ by the Dominated Convergence Theorem. Therefore by the symmetry of the heat kernel and the semigroup property (Chapman-Kolmogorov equations) we have

\begin{align*}
& \int_{M} \int_{M} p_{t}\left( x, y \right)  d\mu\left( y \right)d\mu\left( x \right)=\int_{X} p_{2t}\left( x, x \right) d\mu\left( x \right)
\\
& =\int_{M} \sum_{n=-\infty}^{\infty} \sum_{k=0}^\infty(2k+|n|+1)e^{-(4k(k+|n|+1)+2|n|)2t} d\mu\left( x \right) < \infty.
\end{align*}

\end{proof}

Next we address some properties of the spectrum  associated to the heat kernel given in Assumption \ref{Ass.HeatKernel}.
For $k,n \geqslant 0$ we denote 
\begin{align}\label{eq-lam-nk}
\lambda_{k,n}=-(4k(k+n+1)+2|n|)
\end{align}
 and
\begin{align}\label{eq-p-nk}
& p_{k, n}(x,y)=2(2k+n+1)\cos (n\theta(x,y))(\cos r(x,y))^nP_{k}^{0, n}(\cos 2r(x,y)), n \not=0, 
\\
& p_{0, 0}(x,y)=1. \notag
\end{align}
Then the heat kernel in Assumption~\ref{Ass.HeatKernel} can be written as 
\begin{equation}\label{eq-spec-decomp}
p_t(x,y)=\sum_{n=0}^{\infty} \sum_{k=0}^\infty e^{\lambda_{k,n}t}p_{k, n}(x,y),
\end{equation}
since the series converges absolutely.

\begin{lemma}\label{lemma-spectrum}
Suppose $\left( M, r, \theta, \mu \right)$ satisfies Assumption~\ref{Ass.HeatKernel} and let $L$ be the associated self-adjoint operator. Then $L$ has discrete spectrum $\{\lambda_{k,n}\}_{k, n\geqslant 0}$, where $\lambda_{k,n}$ is given as in \eqref{eq-lam-nk}. In particular, we have

\begin{equation}\label{eq-eign-fun}
Lp_{k, n}(x, \cdot)=\lambda_{k, n}p_{k, n}(x, \cdot).
\end{equation}

\end{lemma}
\begin{proof}
For any $f\in L^2(M,\mu)$ we have

\[
\left( P_{t}f \right)\left( x \right)=\int_{M} p_{t}\left( x, y \right) f\left( y \right)d\mu\left( y \right)=\int_{-\infty}^{\infty} e^{t \lambda} d E\left( \lambda\right) f,
\]
where $E=E^{L}$ is the (unique) projection-valued measure (spectral measure) on $\mathcal{B}\left( \mathbb{R} \right)$ for the self-adjoint operator $L$. By Proposition \ref{p.HeatKernel} we know that $P_{t}$ is a Hilbert-Schmidt integral operator for any $t>0$, and therefore $P_{t}$  is compact and has a discrete spectrum. Then the positive Borel measure on $\mathbb{R}$ corresponding to the spectral measure $E$ is denoted by $\nu_{f}\left( A \right):=\langle E\left( A\right) f, f \rangle= \langle E\left( A\right)^{2} f, f \rangle=\langle E\left( A\right) f, E\left( A\right)f \rangle \geqslant 0$ for any $A \in \mathcal{B}\left( \mathbb{R} \right)$, and it is supported on a discrete subset of $\mathbb{R}$ . Then

\begin{align}\label{e.EigenfunctionExpansion}
& \langle P_{t}f, f \rangle=\int_{-\infty}^\infty e^{t\lambda}d\nu_{f}\left( \lambda \right)=\int_{M} \int_{M} p_{t}\left( x, y \right) f\left( x \right) f\left( y \right)d\mu\left( y \right)d\mu\left( x \right)
\notag
\\
& =\sum_{k,n=0}^\infty e^{\lambda_{k,n}t}\int_{M\times M}p_{k,n}(x,y)f(x)f(y)d\mu(x)d\mu(y) < \infty.
\end{align}

As $P_{t}$ has a discrete spectrum, so does $L$, and Equation~\eqref{e.EigenfunctionExpansion} implies that $\left\{ \lambda_{k,n}\right\}_{k, n=0}^{\infty}$ are eigenvalues for $L$ with finite-dimensional eigenspaces. 
We denote by $E_{k,n}:=\mathrm{Ker}(L-\lambda_{k,n}Id)$ the eigenspace in $L^2(M, \mu)$ that corresponds to $\lambda_{k,n}$, and let $P_{k,n}:L^2(M, \mu) \longrightarrow E_{k,n}$ be the projection operator. Then $P_{k,n}$ is an integral operator with the kernel $p_{k,n}$, namely for any $f\in L^2(M, \mu)$ we have that 
\begin{equation}\label{eq-P-kn}
P_{k,n} f(x)=\int_M  p_{k,n}(x,y)f(y)d\mu (y).
\end{equation}
Using the fact that $P_{k,n}$ commutes with $L$ and similar arguments as in \cite{CarronTewodrose2022} we obtain that $p_{k,n}$ is in the domain of $L$. Therefore it is an eigenfunction such that \eqref{eq-eign-fun} holds.
\end{proof}

\begin{remark}
One can easily verify that the Dirichlet form associated with the heat kernel $p_t(x,y)$ is local, by combining the small time estimate of $p(t,r, \theta)\asymp C\exp(-C' t^2)$ for $t$ small enough (see \cite{BaudoinBonnefont2009}) and the characterization theorem of localness in \cite{grigor2008dichotomy}.
\end{remark} 

\section{The first isometry}\label{sec-isom-1}

In this section we address the first isometry as described in the introduction. 

\subsection{The metric space $(M,\delta)$}
First we induce a metric $\delta$ on $M$ from the two pseudo-metrics $r$ and $\theta$ in the Assumption \ref{Assump1}. Note from \eqref{eq-p-nk} we know that the first non-constant eigenfunction corresponding to an eigenvalue  is 
\[
p_{0, 1}(x,y)=4\cos (\theta(x,y))(\cos r(x,y)).
\]
We define the function $\delta: M\times M\to \R_{\ge0}$ by
\begin{equation}\label{e.delta}
\delta(x,y):=\arccos(\cos r(x,y)\cos\theta(x,y)).
\end{equation}
In the lemma below we show that $\delta$ in fact defines a metric on $M$. 

\begin{lemma}\label{lemma-delta}
Let $M$ be a space satisfying Assumption \ref{Assump1}, and let $\delta: M\times M\to \R_{\geqslant 0}$ be given by \eqref{e.delta}.
Then $\delta$ is a distance on $M$.
\end{lemma}
\begin{proof}
First note that $\delta(x,y)=0$ implies that $r(x,y)=\theta(x,y)=0$. From the definition of $r$ and $\theta$ we know that  $x=y$. Clearly $\delta$ is symmetric due to the symmetries of $r$ and $\theta$. We are left to show that $\delta$ satisfies the triangular inequality. 

Consider any $x, y, z\in M$, we denote by $r_1=r(x,y)$, $r_2=r(y,z)$, $r_3=r(x,z)$ and $\theta_1=\theta(x,y)$, $\theta_2=\theta(y,z)$, $\theta_3=\theta(x,z)$. Since $r$ and $\theta$ are pseudo-metrics we know that they satisfy the triangular inequality hence

\begin{equation}\label{eq-trig}
r_1+r_2\geqslant r_3,\quad \theta_1+\theta_2\geqslant \theta_3.
\end{equation}
Let $\delta_1=\delta(x,y)$, $\delta_2=\delta(y,z)$, $\delta_3=\delta(x,z)$. We want to show that $\delta_1+\delta_2\geqslant \delta_3$. We divide the discussion into 3 cases.

(1) If $\delta_1+\delta_2\in [\pi, 2\pi)$, since $\delta_3\in [0,\pi)$, inequality holds.

(2) If $\delta_1+\delta_2\in [0,\pi)$, and $\delta_1,\delta_2\in [0,\frac{\pi}2)$. From \eqref{e.delta} we know that $\theta_1\in [0,\delta_1 ] $ and  $\theta_2\in[0,\delta_2 ]$.
It suffices to show that $\cos(\delta_1+\delta_2)\leqslant \cos \delta_3$. We have
\begin{align}\label{eq-equality-1}
\cos(\delta_1+\delta_2)=\cos r_1\cos \theta_1\cos r_2\cos\theta_2-\sqrt{(1-\cos^2r_1\cos^2 \theta_1)(1-\cos^2r_2\cos^2 \theta_2)}.
\end{align}
On the other hand since  $\theta_3\leqslant \theta_1+\theta_2<\frac\pi2$, by \eqref{e.delta} and \eqref{eq-trig} we also have
\begin{align*}
\cos\delta_3&=\cos r_3\cos \theta_3\geqslant \cos (r_1+r_2)\cos (\theta_1+\theta_2)\\
&=(\cos r_1\cos r_2-\sin r_1\sin r_2)(\cos \theta_1\cos\theta_2-\sin \theta_1\sin \theta_2).
\end{align*}
Therefore it is enough to show that 
\begin{align*}
&\sin r_1\sin r_2\sin \theta_1\sin \theta_2-\cos r_1\cos r_2 \sin \theta_1\sin \theta_2-\sin r_1\sin r_2\cos \theta_1\cos\theta_2 \\
&\quad\geqslant -\sqrt{(1-\cos^2r_1\cos^2 \theta_1)(1-\cos^2r_2\cos^2 \theta_2)},
\end{align*}
which is equivalent to 
\begin{align}\label{eq-mid-ineq}
\sqrt{(1-\cos^2r_1\cos^2 \theta_1)(1-\cos^2r_2\cos^2 \theta_2)} \geqslant \sin r_1\sin r_2\cos(\theta_1+\theta_2)+\cos r_1\cos r_2 \sin \theta_1\sin \theta_2.
\end{align}
Note that $\sin r_1\sin r_2>0$, the right hand side of \eqref{eq-mid-ineq} is bounded from above by 
\[
\sin r_1\sin r_2+\cos r_1\cos r_2 \sin \theta_1\sin \theta_2,
\]
which is non-negative. Therefore we are reduced to show that 
\begin{equation}\label{eq-mid-ineq1}
(1-\cos^2r_1\cos^2 \theta_1)(1-\cos^2r_2\cos^2 \theta_2) \geqslant ( \sin r_1\sin r_2+\cos r_1\cos r_2 \sin \theta_1\sin \theta_2)^2.
\end{equation}
Note that 
\[
1-\cos^2r_1\cos^2 \theta_1=\sin ^2r_1+\cos^2r_1\sin^2\theta_1,\quad 1-\cos^2r_2\cos^2 \theta_2=\sin ^2r_2+\cos^2r_2\sin^2\theta_2
\]
The left hand side of \eqref{eq-mid-ineq1} is then 
\begin{align*}
&(\sin ^2r_1+\cos^2r_1\sin^2\theta_1)(\sin ^2r_2+\cos^2r_2\sin^2\theta_2)\\
&\quad=\sin ^2r_1\sin ^2r_2+\cos^2r_1\sin^2\theta_1\cos^2r_2\sin^2\theta_2+\sin ^2r_1\cos^2r_2\sin^2\theta_2+\cos^2r_1\sin^2\theta_1\sin ^2r_2\\
&\quad\geqslant \sin ^2r_1\sin ^2r_2+\cos^2r_1\sin^2\theta_1\cos^2r_2\sin^2\theta_2+2\sin r_1\cos r_2\sin\theta_2\cos r_1\sin \theta_1\sin r_2.
\end{align*}
We then obtain the desired  \eqref{eq-mid-ineq1} by realizing that its right hand side matched with the right hand side of the above inequality.  

(3) If $\delta_1+\delta_2\in[0,\pi)$ and $\delta_1\in [0,\frac\pi2 )$, $\delta_2\in [\frac\pi2, \pi)$, then $\theta_1\in[0,\delta_1]$ and $\theta_2\in [\delta_2, \pi)$. Again we just need to show that $\cos(\delta_1+\delta_2)\leqslant \cos \delta_3$. If $\delta_3<\frac\pi2$, then inequality holds. If $\delta_3\in (\frac\pi2, \pi]$, then $\theta_3\in [\delta_3, \pi]$. Similarly as in case (2), we still have the equality \eqref{eq-equality-1}. Next we use the fact $r_3\geqslant |r_1-r_2|$ to estimate 
\[
\cos r_3\leqslant (\cos r_1\cos r_2+\sin r_1\sin r_2).
\]
Hence 
\begin{align*}
\cos\delta_3&=\cos r_3\cos \theta_3\geqslant \cos r_3\cos (\theta_1+\theta_2)\\
&\ge(\cos r_1\cos r_2+\sin r_1\sin r_2)(\cos \theta_1\cos\theta_2-\sin \theta_1\sin \theta_2).
\end{align*}
Therefore it suffices to show that 
\begin{align*}
&-\sin r_1\sin r_2\sin \theta_1\sin \theta_2-\cos r_1\cos r_2 \sin \theta_1\sin \theta_2+\sin r_1\sin r_2\cos \theta_1\cos\theta_2 \\
&\quad\geqslant -\sqrt{(1-\cos^2r_1\cos^2 \theta_1)(1-\cos^2r_2\cos^2 \theta_2)},
\end{align*}
which is equivalent to 
\begin{align*}
&\sqrt{(1-\cos^2r_1\cos^2 \theta_1)(1-\cos^2r_2\cos^2 \theta_2)}
\geqslant -\sin r_1\sin r_2\cos(\theta_1+\theta_2)+\cos r_1\cos r_2 \sin \theta_1\sin \theta_2.
\end{align*}
We then just need to show that 
\[
\sqrt{(1-\cos^2r_1\cos^2 \theta_1)(1-\cos^2r_2\cos^2 \theta_2)}
\geqslant \sin r_1\sin r_2+\cos r_1\cos r_2 \sin \theta_1\sin \theta_2.
\]
using the same computation as in case (2).

The proof is then complete.
\end{proof}
In the proposition below we show that the metric space $(M,\delta)$ is in fact complete.
\begin{proposition}
Let $(M,\delta)$ be given as in Lemma \ref{lemma-delta}. Then it is a complete metric space.
\end{proposition}
\begin{proof}
 Consider a Cauchy sequence $\{x_n\}_{n\ge1}$ in $(M,\delta)$ such that
\[
\lim_{m,n\to\infty}\delta(x_n,x_m)\to 0.
\]
From \eqref{e.delta} we know that 
\[
\lim_{m,n\to\infty}r(x_m,x_n)=\lim_{m,n\to\infty} \theta(x_n,x_m)= 0.
\]
From Assumption \ref{Assump1} we know that $M$ is complete with respect to $r$ and $\theta$. Hence there exists a $x\in M$ such that $r(x_n,x)\to0$ and $\theta(x_n,x)\to0$ as $n\to\infty$, which implies that $\lim_{n\to\infty}\delta(x_n,x)=0$. 
\end{proof}

\subsection{Dirichlet form and heat kernel on $(M, \delta, \mu)$}

We first construct a heat kernel on the metric measure space $(M, \delta, \mu)$ that is essential for the later construction of the isometry $ (M,\delta)\to (\cS^3, d_{\cS^3})$.  First let us recall the Poisson summation formula. 

\begin{definition}\label{def-q}
Let $q: [0,\infty)\times \R\to \R_{\ge0}$ be a function such that
\begin{align}\label{eq-q-kernel-1}
q(t, x):= \frac{\sqrt{\pi}e^t}{4t^{3/2}}\frac{1}{\sqrt{1- x^2}}\sum_{k\in \mathbb{Z}}(\arccos x+2k\pi) e^{-\frac{(\arccos x+2k\pi)^2}{4t}}
\end{align}
It follows from the Poisson summation formula that $q$ also admits the following expression
\begin{align}\label{eq-q-kernel-2}
q(t, \cos x)= \frac{\sqrt{\pi}e^t}{4t^{3/2}}\frac{x}{\sin x}\left( 1+ 2\sum_{k=1}^\infty e^{-\frac{k^2\pi^2}{t}}\left(\cosh\frac{k\pi}{t}+2k\pi\frac{\sinh\frac{k\pi\delta}{t}}{\delta}\right)\right).
\end{align}
It is an obvious fact that $q(t,\cdot)$ admits an analytic extension for $\delta\in \mathbb{C}\setminus (-\infty, -1]$. We then have that for any $x\in [1,\infty)$ and $t>0$,
\begin{align}\label{eq-q-kernel-3}
q(t,x)= \frac{\sqrt{\pi}e^t}{4t^{3/2}}\frac{\cosh^{-1}x}{\sqrt{x^2-1}}e^{\frac{(\cosh^{-1}x)^2}{4t}}(1+R(t,x)),
\end{align}
where $|R(t,x)|\leqslant Ce^{-\frac{c}{t}}$ for some constants $c, C>0$. For any $x\in (-1+\epsilon, 1]$ for some $\epsilon>0$ and any $t>0$,
\begin{align}\label{eq-q-kernel-4}
q(t,x)= \frac{\sqrt{\pi}e^t}{4t^{3/2}}\frac{\arccos x}{\sqrt{1-x^2}}e^{-\frac{(\arccos x)^2}{4t}}(1+R'(t,x)),
\end{align}
where $|R'(t,x)|\leqslant C_\epsilon e^{-\frac{c_\epsilon}{t}}$ for some constants $c_\epsilon, C_\epsilon>0$ that depend only on $\epsilon$.
\end{definition}

The lemma below reveals the relation between the function $q$ as defined above and the expression $p$ as given in \eqref{eq-kernel}.
\begin{lemma}\label{lemma-kernel-equi}
Let $p:[0,\infty)\times [0, \pi/2) \times [-\pi, \pi)\to \R$ be given as in \eqref{eq-kernel},
then for any $r,\theta\in [0, \pi/2) \times [-\pi, \pi)$ and $t\ge0$ we have that 
\begin{equation}\label{eq-kernel-int}
p(t,r,\theta)=\frac{1}{2\pi^2}\frac{1}{\sqrt{4\pi t}}\int_{-\infty}^{+\infty} e^{-\frac{(y+i\theta)^2}{4t}}q(t,\cos r\cosh y)dy.
\end{equation}
\end{lemma}
\begin{proof}
First from \cite{BaudoinBonnefont2009} one can verify that  $p(t,r,\theta)$ solves the following partial differential equation,
\begin{equation}\label{eq-heat-eqn}
\left(\frac{\partial}{\partial t}-\Delta\right)u(t,r,\theta)=0
\end{equation}
for all $(t, r,\theta)\in [0,\infty)\times [0, \pi/2) \times [-\pi, \pi)$ and satisfies that $\lim_{t\to0}p(t,r,\theta)=\delta_{(0,0)}(r,\theta)$
where $\delta_{(0,0)}(\cdot)$ is the dirac mass at $(0,0)$, and
\begin{equation}\label{eq-Delta}
\Delta=\frac{\partial^2}{\partial r^2}+2\cot 2r\frac{\partial}{\partial r}+\tan^2 r\frac{\partial^2}{\partial \theta^2}.
\end{equation}
On the other hand, we know that $q(t,\delta)$ solves the partial differential equation 
\[
\left(\frac{\partial}{\partial t}-\Delta_R\right)v(t,\delta)=0
\]
for all $(t,\delta)\in (0,\infty)\times [0, \pi)$ and satisfies that  $\lim_{t\to0}q(t,\cdot)=\delta_{(0)}(\cdot)$ where
\[
\Delta_R=\frac{\partial^2}{\partial\delta^2}+6\cot\delta\frac{\partial}{\partial \delta}.
\]
Consider the function $\varrho: [0, \pi/2) \times [-\pi, \pi)\to [0, \pi)$ such that for all $(r,\theta)\in [0, \pi/2)\times [-\pi, \pi)$,
\[
\varrho(r,\theta)=\arccos (\cos r\cos \theta).
\]
It is then easy to verify that $\Delta_R=\varrho_*(\Delta+\frac{\partial^2}{\partial\theta^2})$. Namely for any $f\in C^2([0, \pi),\R)$ and $(r,\theta)\in [0, \pi/2)\times [-\pi, \pi)$, we have that 
\[
(\Delta+\frac{\partial^2}{\partial\theta^2})(f\circ \varrho)(r,\theta)=(\Delta_R f)(\varrho(r,\theta)).
\] 
Now let $g$ be a function on $[0,\infty)\times [0, \pi/2) \times [-\pi, \pi)$ which takes the form of \eqref{eq-kernel-int}. It is shown in \cite{BaudoinBonnefont2009} that $g(t,r,\theta)$ is a fundamental solution of \eqref{eq-heat-eqn}. 

By the uniqueness of solution to the equation \eqref{eq-heat-eqn} we then obtain that $p_t(r,\theta)=g_t(r,\theta)$ for all $(t,r,\theta)\in [0,\infty)\times [0,\frac\pi2)\times[-\pi,\pi)$. 
\end{proof}
The equality \eqref{eq-kernel-int} has a conversed version as presented below.
\begin{corollary}
Let $p(t,r,\theta)$ and $q(t,\delta)$ be as in Lemma \ref{lemma-kernel-equi}. Let $\zeta(t,\varphi):=\frac{1}{\sqrt{4\pi t}}e^{-\frac{\varphi^2}{4t}}$ be the Gaussian kernel on $\R$. Then we have
\begin{equation}\label{eq-qt-conv}
q(t,\arccos (\cos r\cos \theta))=\int_\R \zeta(t, \theta-\varphi)p(t,r,\varphi)d\varphi.
\end{equation}
\end{corollary}
\begin{proof}
We first show that the right hand side of \eqref{eq-qt-conv}, which we denote by $v(t,r,\theta)$, solves the differential equation
\[
\left(\frac{\partial}{\partial t}-\Delta_R\right)v(t,r,\theta)=0.
\]
Since
\begin{align*}
\Delta_R v(t, r, \theta)&=\left(\frac{\partial^2}{\partial r^2}+2\cot 2r\frac{\partial}{\partial r}+\sec^2 r\frac{\partial^2}{\partial \theta^2}\right)\int_\R \zeta(t, \theta-\varphi)p(t,r,\varphi)d\varphi \\
&=\int_\R \sec^2 r  \frac{\partial^2}{\partial \theta^2}\zeta(t, \theta-\varphi)p(t,r,\varphi)d\varphi + \int_\R \zeta(t, \theta-\varphi) \left(\frac{\partial^2}{\partial r^2}+2\cot 2r\frac{\partial}{\partial r}\right)p(t,r,\varphi) d\varphi.
\end{align*}
Using the  fact that $p(t,r,\theta)$ satisfies \eqref{eq-heat-eqn} we have
\[
\left(\frac{\partial^2}{\partial r^2}+2\cot 2r\frac{\partial}{\partial r}\right)p(t,r,\varphi)=\frac{\partial}{\partial t}p(t,r,\varphi)-\tan^2 r \frac{\partial^2}{\partial \varphi^2} p(t,r,\varphi).
\]
Using integral by parts and combining with the fact that $\zeta$ is a Gaussian kernel we obtain that
\begin{align}\label{eq-v-heat-eqn}
\Delta_R v(t, r, \theta)
&=\int_\R \bigg(  \frac{\partial}{\partial t}\zeta(t, \theta-\varphi)p(t,r,\varphi) + \zeta(t, \theta-\varphi)\frac{\partial}{\partial t}p(t,r,\varphi)  \bigg)  d\varphi \notag \\
&=\frac{\partial}{\partial t} v(t, r, \theta).
\end{align}
Moreover, we already know that $\lim_{t\to0}q(t,\cdot)=\delta_{(0)}(\cdot)$, $\lim_{t\to0}\zeta(t,\cdot)=\delta_{(0)}(\cdot)$ and $\lim_{t\to0}p(t,\cdot, \cdot)=\delta_{(0,0)}(\cdot,\cdot)$. Plugging them into \eqref{eq-v-heat-eqn} we can easily obtain that
\[
\lim_{t\to0}q(t,\arccos (\cos r\cos \theta))=\delta_{(0,0)}(r,\theta)=\lim_{t\to0}v(t,r,\theta).
\]
\end{proof}
We now define a function on $M^2\times \R_{\ge0}$ by convolving $q: [0,\infty)\times \R\to \R_{\ge0}$ and $\delta: M\times M\to \R_{\ge0}$ as given in \eqref{e.delta}, let
\begin{equation}\label{eq-Riem-kernel}
q_t(x,y):=q(t,\delta(x,y)).
\end{equation}
In the proposition below we show that $q_t(x,y)$ is in fact a heat kernel on $(M, \delta, \mu)$.

\begin{proposition}\label{lemma-q-t}
Let $q_t(x,y)$, $t\ge0$, $x,y\in M$ be as given in \eqref{eq-Riem-kernel}. Then 
\begin{equation}\label{eq-qt-explicit}
q_t(x,y)=\sum_{n=-\infty}^{\infty} \sum_{k=0}^\infty e^{\lambda'_{k,n} t}e^{in\theta(x,y)}\Phi_{k,n}(r(x,y))
\end{equation}
where $\lambda'_{k,n}=-(4k(k+|n|+1)+2|n|+n^2)$ and $\Phi_{k,n}(r)=4k(k+|n|+1)+2|n|(\cos r)^{|n|}P_k^{0, |n|}(\cos 2r)$.
Moreover, $q_t(x,y)$ is a heat kernel on  $(M, \delta, \mu)$. Precisely, it satisfies the conditions (i)-(v) in Definition~\ref{d.HeatKernel}.  
\end{proposition}
\begin{proof}
To obtain \eqref{eq-qt-explicit}, we just need to combining \eqref{eq-qt-conv} and \eqref{eq-kernel}, and obtain that
\begin{align}\label{eq-qt-exp}
q_t(x,y)&= \int_\R \zeta(t, \varphi)p(t,r(x,y),\theta(x,y)-\varphi)d\varphi \notag \\
&=\sum_{n=-\infty}^{\infty} \sum_{k=0}^\infty e^{\lambda_{k,n} t}e^{in\theta(x,y)}\Phi_{k,n}(r(x,y))\left( \int_\R \zeta(t, \varphi) e^{-in\varphi}d\varphi \right)
\end{align}
where $\lambda_{k,n}=-(4k(k+|n|+1)+2|n|)$, $\Phi_{k,n}(r)=a_{k,n} (\cos r)^{|n|}P_k^{0, |n|}(\cos 2r)$, and $a_{k,n}=2k+|n|+1 $.
Note that
\[
 \int_\R \zeta(t, \varphi) e^{-in\varphi}d\varphi = \int_\R \zeta(t, \varphi) e^{-i|n|\varphi}d\varphi=e^{-n^2t}  \quad\text{for all }n\in \mathbb{Z},
\]
we then obtain the desired expression.

Next to show that $q_t(x,y)$ is a heat kernel, firstly (i) and (ii) easily follow from \eqref{eq-Riem-kernel}. To see (iii), we rewrite $q_t$ as  
\begin{align}\label{eq-qt-spect}
q_t(x,y)
&=\sum_{n=0}^{\infty} \sum_{k=0}^\infty p_{k,n} (x,y)e^{\lambda'_{k,n} t}
\end{align}
where $p_{k,n}(x,y)$ are as given in \eqref{eq-p-nk}. Therefore we have that
\begin{equation}\label{eq-qt-iii}
\int_Mq_t(x,y)d\mu(y)=\int_M  \sum_{n=0}^{\infty} \sum_{k=0}^\infty p_{k,n} (x,y)e^{\lambda'_{k,n} t}d\mu(y).
\end{equation}
 Using the fact that $p_{k,n}(x,y)$ form an orthonormal basis for $L^2(\mu)$ and $p_{0,0}(x,y)=1$,  we have that 
 \[
 \int_M p_{k,n}(x,y)d\mu(y)= \int_M p_{k,n}(x,y)p_{0,0}(x,y)d\mu(y)=\delta_{0,0}(k,n).
 \]
 Hence 
 \[
  \int_Mq_t(x,y)d\mu(y)=1.
 \]
 Next we show (iv). By \eqref{eq-qt-spect} we know that
\begin{align}\label{eq-qt-iv}
&\int_Mq_s(x,z)q_t(z,y)d\mu(z)\notag\\
&=\int_M \sum_{n,m=0}^{\infty} \sum_{k,\ell=0}^\infty p_{k,n} (x,z)p_{\ell,m} (z,y)e^{(\lambda'_{k,n}s+\lambda'_{\ell,m} t)}d\mu(z)
 \end{align}
Moreover, due to the fact that $p_t(x,y)$ is a heat kernel and satisfies the Chapman-Kolmogorov equation, we have that
 \[
\int_Mp_s(x,z)p_t(z,y)d\mu(z)=p_{s+t}(x,y).
 \]
which implies that for all $s,t>0$ and $x,y\in M$,
\[
\int_M \sum_{n,m=0}^{\infty} \sum_{k,\ell=0}^\infty p_{k,n} (x,z)p_{\ell,m} (z,y)e^{(\lambda_{k,n}s+\lambda_{\ell,m} t)}d\mu(z)
=\sum_{n=0}^{\infty} \sum_{kl=0}^\infty p_{k,n} (x,y)e^{\lambda_{k,n}(s+t)}
\]
Identifying each terms on both sides, we obtain that 
\begin{equation}\label{eq-p-nk-orthog}
\int_M p_{k,n} (x,z)p_{\ell,m} (z,y)d\mu(z)=\delta_{k,n}(\ell,m)p_{n,k} (x,y), \quad\text{for all }n,k,m,\ell \ge0.
\end{equation}
Plugging these back into \eqref{eq-qt-iv} we then have that
\begin{align}\label{eq-qt-iv-a}
&\int_Mq_s(x,z)q_t(z,y)d\mu(z)= \sum_{n=0}^{\infty} \sum_{k=0}^\infty p_{k,n} (x,y)e^{\lambda_{k,n}(s+ t)}e^{-sn^2}e^{-tn^2}d\mu(z)
 \end{align}
We can then obtain the conclusion. 

 Lastly, let us prove (v). Again using \eqref{eq-qt-spect} we have for any $f\in L^2(\mu)$ that
\begin{align*}
 \int_M q_t(x,y)f(y)d\mu(y)&=\int_M \sum_{n=0}^{\infty} \sum_{k=0}^\infty p_{k,n} (x,y)e^{\lambda'_{k,n} t} f(y)d\mu(y) \\
 &= \sum_{n=0}^{\infty} \sum_{k=0}^\infty  C_{k,n}(x)  e^{\lambda'_{k,n} t}
 \end{align*}
 where $C_{k,n}(x) =\int_M  p_{k,n} (x,y) f(y)d\mu(y) $. By \eqref{eq-p-nk-orthog} we know that $C_{k,n}(x)$ are orthogonal in $L^2(\mu)$.
 Hence 
 \[
  \int_M \left( \int_M q_t(x,y)f(y) d\mu(y) -f(y)\right)^2d\mu(y) =  \sum_{n=0}^{\infty} \sum_{k=0}^\infty  \|C_{k,n}(x)\|^2_{L^2} ( e^{\lambda'_{k,n} t}-1)^2
 \]
 Clearly for each $n,k\ge0$, $\|C_{k,n}(x)\|^2_{L^2} ( e^{\lambda'_{k,n} t}-1)^2\to 0$ as $t\to0$. Moreover, since $ \|C_{k,n}(x)\|^2_{L^2} ( e^{\lambda'_{k,n} t}-1)^2\leqslant  \|C_{k,n}(x)\|^2_{L^2}$ and 
 \[
  \sum_{n=0}^{\infty} \sum_{k=0}^\infty  \|C_{k,n}(x)\|^2_{L^2} =\|f\|_{L^2}^2<\infty,
 \]
 by Dominated Convergence Theorem we know that 
 \[
  \sum_{n=0}^{\infty} \sum_{k=0}^\infty  \|C_{k,n}(x)\|^2_{L^2} ( e^{\lambda'_{k,n} t}-1)^2\to 0,\quad \text{as }t\to0.
 \]
 We thus complete the proof.
\end{proof}

It is a standard argument   to construct a symmetric Dirichlet form $\mathcal{E}'$ on  $(M, \delta, \mu)$ associated to the heat kernel $q_t(x,y)$, and thus a self-adjoint operator $L'$. 

\begin{corollary}\label{prop-qt-LDP}
The metric measure space $(M,\delta,\mu)$ can be endowed with a symmetric Dirichlet form $\mathcal{E}'$ admitting a heat kernel $q_t$. Moreover, for all $x,y\in M$
\begin{align}\label{eq-LDP}
\lim_{t\to0}-4t \log q_t(x,y)=\delta^2(x,y).
\end{align}
\end{corollary}
\begin{proof}
The first assertion follows from classical arguments as presented in Section \ref{sec-Diri-M}. To obtain \eqref{eq-LDP}, we combine \eqref{eq-q-kernel-4} and \eqref{eq-Riem-kernel} and obtain that
\[
\lim_{t\to0}-4t \log q_t(x,y)=\lim_{t\to0}-4t \log q(t, \delta(x,y))=\delta^2(x,y).
\]
\end{proof}

Next we will show that $ (M,\delta,\mu)$ is a length space. First let us address the  locally compactness of it.
\begin{lemma}\label{lemma-local-cpmt}
Let $(M, \delta, \mu)$ be as given in  Corollary \ref{prop-qt-LDP}. Then $(M,\delta)$ is locally compact. 
\end{lemma}
\begin{proof}
It is equivalent to show that all closed ball in $M$ is compact. Let $B\subset M$ be a closed ball of radius $R\in(0,1)$ and centered at $x_0\in M$. Take any sequence $\{x_k\}_{k\ge1}$, it suffices to show that $\{x_k\}$ is a Cauchy sequence with respect to $\delta$, namely $\delta(x_k,x_\ell)\to0$ as $k,\ell\to\infty$.

Let $u_k(\cdot):=q_t(x_k,\cdot)$ for all $k\ge1$. Clearly for fixed $t\in (0,1/2)$, by \eqref{eq-q-kernel-4} we have that
\[
\int_M u_k^2d\mu=q_{2t}(x_k, x_k)=\frac{\sqrt{\pi}e^{2t}}{4(2t)^{3/2}}(1+o(t)).
\]
This implies that $\{u_k\}$ is bounded in $L^2(\mu)$ hence weakly converges to a function in $L^2(\mu)$, say $u_\infty$. Next let
\begin{equation}\label{eq-vk}
v_k(\cdot):=\int_{M}u_k(z)q_t(z,\cdot)d\mu (z)=q_{2t}(x_k, \cdot).
\end{equation}
Using Chapman-Kolmogorov we know that 
\[
\int_{M}v_k^2(x)d\mu(x)=\int_{M}q_{2t}^2(x_k, x)d\mu(x)=q_{4t}(x_k, x_k)
\]
hence $v_k\in L^2(\mu)$. Moreover, from \eqref{eq-vk} and using bounded convergence theorem we know that for all $x\in M$
\[
\lim_{k\to\infty} v_k(x)=\lim_{k\to\infty} \int_{M}u_k(z)q_t(z,x)d\mu (z)= \int_{M}u_\infty(z)q_t(z,x)d\mu (z).
\]
We denote $v_\infty (x):=  \int_{M}u_\infty(z)q_t(z,x)d\mu (z)$, then by  dominated convergence theorem we know that $v_k$ converges to $v_\infty$ in $L^2(\mu)$. This implies that $\{v_k\}$ is a Cauchy sequence in $L^2(\mu)$. Note that
\begin{align}\label{eq-vk-Cauchy}
\|v_k-v_\ell\|_{L^2}^2&=\|v_k\|_{L^2}^2+\|v_\ell\|_{L^2}^2-2\int_M v_kv_\ell d\mu, \notag \\
&= q_{4t}(x_k, x_k)+q_{4t}(x_\ell, x_\ell)-2q_{4t}(x_k, x_\ell).
\end{align}
Moreover from \eqref{eq-q-kernel-4} we have for any $x,y\in B$,
\[
q_{4t}(x,y)= \frac{\sqrt{\pi}e^{4t}}{4(4t)^{3/2}}\frac{ \delta(x,y)}{\sin\delta(x,y)}e^{-\frac{\delta^2(x,y)}{8t}}(1+R'(t,\delta(x,y))).
\]
Plugging the above equality into \eqref{eq-vk-Cauchy} we have 
\[
\|v_k-v_\ell\|_{L^2}^2=\frac{\sqrt{\pi}e^{4t}}{4(4t)^{3/2}}\left(2+2R'(t,0)-2\frac{ \delta(x_k,x_\ell)}{\sin\delta(x_k,x_\ell)}e^{-\frac{\delta^2(x_k,x_\ell)}{8t}}(1+R'(t,\delta(x_k,x_\ell))) \right).
\]
Using the fact that $|R'(t,x)|\leqslant  Ce^{-\frac{c}{t}}$ for some $C, c>0$ we know that for all $t$ sufficiently small,
 \[
\|v_k-v_\ell\|_{L^2}^2\to 0 \quad\text{implies}\quad \frac{ \delta(x_k,x_\ell)}{\sin\delta(x_k,x_\ell)}e^{-\frac{\delta^2(x_k,x_\ell)}{8t}}\to 1.
\]
At last we can conclude that $\{x_k\}$ is a Cauchy sequence with respect to $\delta$.
 \end{proof}
\begin{proposition}\label{prop-lengthy}
Let $(M, \delta, \mu)$ be as given in  Corollary \ref{prop-qt-LDP}, then it is a geodesic space. 
\end{proposition}
\begin{proof}
We use similar argument as in \cite{CarronTewodrose2022}. We just need to show that for any $x,y\in M$, there exists an $m\in M$ such that 
\[
\delta(x,m)=\delta(m,y)=\frac12\delta(x,y).
\]
First consider the heat kernel $q_t$ on $M$. By Lemma \ref{lemma-q-t} (iv) we know that 
\begin{equation}\label{eq-qt-chap-Kol}
\int_M q_t(x,z)q_t(z,y)d\mu(z)=q_{2t}(x,y).
\end{equation}
Using the fact that for sufficiently small $t_0>0$, and all $t\in(0,t_0)$, it holds that 
\[
q_t(x,z)=e^{-\frac{\delta^2(x,z)}{4t}+o(\frac{1}{t})}
\]
for all $x,z\in M$. Hence \eqref{eq-qt-chap-Kol} can be rewritten as 
\[
\int_M e^{-\frac{\delta^2(x,z)+\delta^2(z,y) }{4t}+o(\frac{1}{t})}d\mu(z)=  e^{-\frac{\delta^2(x,y) }{8t}+o(\frac{1}{t})}
\]
Raising both sides to the power $t$ we have
\[
\left\|e^{-\frac{\delta^2(x,z)+\delta^2(z,y) }{4}+o(1)}\right\|_{L^{1/t}(\mu)}=e^{-\frac{\delta^2(x,y) }{8}+o(1)}.
\]
As $t\to0$ we then obtain that 
\[
\inf_{z\in M}\{\delta^2(x,z)+\delta^2(z,y)\}=\frac12\delta^2(x,y)
\]
From Lemma \ref{lemma-local-cpmt} we know that $(M,\delta)$ is locally compact, hence the minimum of the left hand side of the above equality can be attained, say by $m\in M$. That is 
\[
\delta^2(x,m)+\delta^2(m,y)=\frac12\delta^2(x,y).
\]
However, since for any $y\in M$ we have 
\[
\delta^2(x,m)+\delta^2(y,m)\geqslant \frac12 \delta^2(x,y)+\frac12(\delta(x,m)-\delta(m,y))^2.
\]
Combining the above two equalities we then obtain that 
\[
\delta(x,m)=\delta(y,m).
\]
By \cite[Theorem 2.4.26]{BBI01} we can then obtain the conclusion. 
\end{proof}

\subsection{Volume growth and dimension estimate of $(M,\delta,\mu)$ }
In this section, we want to estimate the Hausdorff dimension of $(M,\delta,\mu)$ using information contained in its heat kernel. A standard strategy is to obtain the volume growth estimate for a geodesic ball in it. 

\begin{lemma} \label{lemma-Hasud-dim}
Let $\delta$ be the distance as in Lemma \ref{lemma-delta}, and let $B_\delta(x,\rho)$ denote the ball  centered at $x\in (M,\delta)$ of radius $\rho>0$.  Then 
\[
 \mu(B_\delta(x,\rho))\leqslant C\, \rho^3.
\]
for some constant $C>0$.
\end{lemma}
\begin{proof}
Combine Proposition \ref{lemma-q-t} and \ref{eq-q-kernel-4} we know that there exists a $t_0>0$, $\rho_0>0$ and some $C>0$ such that for all $t\in[0,t_0)$, $\rho\in(0,\rho_0)$, $x\in M$ and $y\in B_\delta(x, \rho)$
\[
\frac{1}{Ct^{3/2}} e^{-\frac{\delta(x,y)^2}{4t}}\leqslant q(t,x,y)\leqslant \frac{C}{t^{3/2}} e^{-\frac{\delta(x,y)^2}{4t}}.
\]
From the lower bound in the above estimate we know that 
\begin{align*}
 \int_{B_\delta(x,\rho)}\frac{e^{-\frac{\delta(x,y)^2}{4t}}}{Ct^{3/2}}d\mu(y)
& \leqslant  \int_{B_\delta(x,\rho)} q_t(x,y)d\mu(y)\leqslant \int_{M} q_t(x,y)d\mu(y) \leqslant 1.
\end{align*}
But the left hand side is bounded from below by
\[
\int_{B_\delta(x,\rho)}\frac{e^{-\frac{\rho^2}{4t}}}{Ct^{3/2}}d\mu(y)=\frac{e^{-\frac{\rho^2}{4t}}}{Ct^{3/2}}\mu(B_\delta(x,\rho)).
\]
Therefore we obtain that for any $0<t<t_0$
\[
\mu(B_\delta(x,\rho))\leqslant \frac{e^{\frac{\rho^2}{4t}}}{C} t^{3/2}. 
\]
Take $t=t_0\rho^2$ such that $t<t_0$ we then obtain that for any $x\in M$ and $\rho\leqslant \rho_0\wedge 1$
\[
\mu(B_\delta(x,\rho))\leqslant C'\,  \rho^3,
\]
for some constant  $C'>0$ that depends on $t_0$. 
\end{proof}

As a direct consequence, we can readily conclude the following lower bound on the Hausdorff dimension for $(M,\delta,\mu)$. 

\begin{proposition}\label{cor-dim-3}
The Hausdorff dimension of $(M,\delta,\mu)$ is at least $3$. 
\end{proposition}

\begin{proof}
The conclusion follows from classical Frostman's Lemma. 
\end{proof}

\subsection{The first isometry}\label{sec-eigen1} From \eqref{eq-qt-explicit} we know that when $k=0$, $n=1$, 
\begin{align}\label{eq-p01}
\lambda_{0,1}=-3,\quad p_{0,1}( \cdot,y)=4\cos \theta(\cdot,y)\cos r(\cdot,y)
\end{align}
and
\[
L\cos \theta(\cdot, y)\cos r(\cdot,y)=-2\cos \theta(\cdot,y)\cos r(\cdot,y).
\]
\begin{lemma}
Consider  the eigenspace $E_{0,1}$  of $L$ correspond to $\lambda_{0,1}=-2$, then $\dim E_{0,1}=4$, and 
\begin{align}\label{eq-E01}
E_{0,1}=\operatorname{Span}\{\cos \theta(x,\cdot)\cos r(x,\cdot), x\in M\}.
\end{align}
\end{lemma}
\begin{proof}
Let $\{\varphi_1,\dots, \varphi_\ell\}$ be continuous functions on $(M,\delta, \mu)$  that form an orthonormal basis of $E_{0,1}$ in $L^2(M, \mu)$. Then for any $f\in L^2(M, \mu)$ and $x\in M$ we have that
\[
P_{0,1}f(x)=4\int_M\cos \theta(x,y)\cos r(x,y)f(y)d\mu(y)=\sum_{i=1}^\ell \left( \int_M \varphi_i(y)f(y)d\mu(y)\right) \varphi_i(x),
\]
where $P_{0,1}$ is the projection operator as defined in \eqref{eq-P-kn}.
Hence for $\mu$-a.s. $x, y\in M$ 
\[
\sum_{i=1}^\ell \varphi_i(x)\varphi_i(y)=4\cos \theta(x,y)\cos r(x,y),
\]
and thus $\sum_{i=1}^\ell \varphi_i(x)^2=4$ for $\mu$-a.s. $x\in M$. 
Integrate both sides with respect to $\mu$ we then obtain that $\ell=4$. 
Moreover, if we denote by
\[
V_1:=\operatorname{Span}\{\cos \theta(x,\cdot)\cos r(x,\cdot), x\in M\}
\]
From \eqref{eq-p01} it is clear that $V_1\subset E_{0,1}$. On the other hand, since $E_{0,1}$ is the image of $P_{0,1}$ in $L^2(M, \mu)$, for any $g\in E_{0,1}$ we have that
\[
g(x)=P_{0,1}g(x)=4\int_M \cos \theta(x,y)\cos r(x,y)g(y)d\mu(y)\in V_{1}.
\]
Hence we have $V_1=E_{0,1}$.
\end{proof}
Now consider the space $\mathcal{D}_1:=\operatorname{Span}\{\delta_x(\cdot), x\in M\}$, where $\delta_x(\cdot)$ denote the Dirac mass at $x$. Clearly $\mathcal{D}_1$ is a subspace of the algebraic dual space $V_1^*$ of $V_1$.   Moreover, note that for any $f\in V_1$ such that $\eta(f)=0$ for all $\eta\in \mathcal{D}_1$ we have $f(x)=0$ for all $x\in M$. This means that
\[
\mathcal{D}_1=V_1^*.
\]
\begin{lemma}
Let $\{x_1,\dots, x_4\}\in M$ be such that $\{\delta_{x_1},\dots, \delta_{x_4}\}$ is a basis of $V_1^*$ dual to the basis $\{\varphi_1,\dots, \varphi_4\}$ in $V_1$. Then for $\mu$-a.s. any $x,y\in M$ we have 
\begin{equation}\label{eq-V-1-metric}
\cos \theta(x,y)\cos r(x,y)=\sum_{i,j=1}^4 \cos \theta(x_i,x_j)\cos r(x_i,x_j)\varphi_i(x) \varphi_j(y).
\end{equation}
\end{lemma}
\begin{proof}
By definition we can write for $\mu$-a.s. any $x\in M$ that 
\begin{equation}\label{eq-mid-3}
\cos \theta(x,\cdot)\cos r(x,\cdot)=\sum_{j=1}^4 \delta_{x_j}\left(\cos \theta(x,\cdot)\cos r(x,\cdot) \right) \varphi_j(\cdot)
=\sum_{j=1}^4\cos \theta(x,x_j)\cos r(x,x_j)\varphi_j(\cdot).
\end{equation}
Moreover, note that for each $j=1,\dots, 4$,
\[
\cos \theta(x,x_j)\cos r(x,x_j) =\cos \theta(x_j,x)\cos r(x_j,x)=\sum_{i=1}^4 \cos \theta(x_i,x_j)\cos r(x_i,x_j)\varphi_i(x),
\]
Plugging it into \eqref{eq-mid-3} we then obtain the conclusion.
\end{proof}

\begin{definition}\label{def-beta-Q}
Clearly $\{\cos \theta(x_i,\cdot)\cos r(x_i,\cdot)\}_{i=1}^4$ forms a basis for $V_1$.
Let
\begin{align}\label{eq-a}
a_{ij}:= \cos \theta(x_i,x_j)\cos r(x_i,x_j)\quad  \text{ for all } 1\leqslant i,j\leqslant 4.
\end{align}
and define a bilinear form on $\R^4$ by letting $\beta_1:\R^4\times \R^4\to\R$ such that
\begin{equation}\label{eq-beta1}
\beta_1 (\zeta,\zeta^{\prime}):=\sum_{i,j=1}^4a_{ij}\zeta_i\zeta_j^{\prime}, \quad \zeta=(\zeta_1,\dots, \zeta_4), \ \zeta^{\prime}=(\zeta^{\prime}_1,\dots, \zeta^{\prime}_4)\in\R^4.
\end{equation}
Let $Q: \R^4\to \R$ be the associated quadratic form, i.e. for any $\zeta\in \R^4$
\begin{equation}\label{eq-Q}
Q(\zeta)=\beta_1(\zeta,\zeta).
\end{equation}
We also define a map $\Phi: M\to \R^4$ by letting
\begin{align}\label{eq-Phi}
\Phi(x):= (\varphi_1(x),\dots, \varphi_4(x)).
\end{align}
\end{definition}

\begin{lemma} 
Let $\Phi:M\to \R^4$ be as given in \eqref{eq-Phi}, and $Q: \R^4\to \R$ the quadratic form that is given in \eqref{eq-Q}. Then $\Phi$ and $Q$ are both injective.
\end{lemma}
\begin{proof}
By \eqref{eq-Q} and the bi-linearity of $\beta_1$ we know that 
\begin{align*}
Q(\Phi(x)-\Phi(y))&=\beta_1(\Phi(x),\Phi(x))+\beta_1(\Phi(y),\Phi(y))-2\beta_1(\Phi(x),\Phi(y)).
\end{align*}
Moreover, combining \eqref{eq-Q}, \eqref{eq-a} and \eqref{eq-V-1-metric} we know that 
\begin{align}\label{eq-beta1-sphere}
\beta_1(\Phi(x),\Phi(x))=\beta_1(\Phi(y),\Phi(y))=1
\end{align}
and
\[
\beta_1(\Phi(x),\Phi(y))=\cos r(x,y)\cos\theta(x,y).
\]
Therefore we obtain that 
\begin{align}\label{eq-Q-delta}
Q(\Phi(x)-\Phi(y))=2-2\cos r(x,y)\cos\theta(x,y).
\end{align}
If $\Phi(x)=\Phi(y)$, then $Q(\Phi(x)-\Phi(y))=0$ which implies that $r(x,y)=\theta(x,y)=0$. Hence we have that $x=y$. This implies the injectivity for both $\Phi$ and $Q$. 
\end{proof}

From the above lemma we know that $\operatorname{Ker}\beta_1=\left\{ 0 \right\}$.
Therefore we can have the following orthogonal decomposition of $\R^4$
\begin{equation}\label{eq-orth-decomp}
\R^4=E_{+}\oplus E_{-}\oplus \operatorname{Ker}\beta_1=E_{+}\oplus E_{-},
\end{equation}
where $E_{+}$ and $E_{-}$ are subspaces of $\R^4$ on which $\beta_1$ is positive definite and negative definite.

\begin{definition}
Denote by $\mathrm{P}_{E_{+}}$ and $\mathrm{P}_{E_{-}}$ the projection map  from $\R^4$ to the subspaces $E_+$ and $E_-$ respectively. We define
\begin{equation}\label{eq-Phi+-}
\Phi_{+}:=\mathrm{P}_{E_{+}}\circ \Phi,\quad \Phi_{-}:=\mathrm{P}_{E_{-}}\circ \Phi.
\end{equation}
Denote by $Q_{+}$ and $Q_{-}$ the restrictions of the quadratic form $Q$ on $E_{+}$ and $E_{-}$, i.e. for any $u\in \R^4$, 
\[
Q_+(u):=Q(\mathrm{P}_{E_{+}}(u)),\quad Q_-(u):=Q(\mathrm{P}_{E_{-}}(u)).
\]
We define the induced distance on $E_+$ from $Q_+$ by setting  for any $u,v\in E_{+}$ that 
\begin{equation}\label{eq-d_Q+}
d_{Q_{+}}(u,v):=\sqrt{Q_{+}(u-v)}.
\end{equation} 
\end{definition}

\begin{lemma}\label{lemma-dim}
Let $\Phi_+$ be as given in \eqref{eq-Phi+-}. Then $\Phi_{+}$ is a bi-Lipschitz embedding of $(M,\delta)$ into $(E_{+}, d_{Q_{+}})$. Moreover we have $\dim E_+=4$, and consequently $\beta_1$ as given in \eqref{eq-beta1} is positive definite on $\R^4$.
\end{lemma}
\begin{proof}
From \eqref{eq-orth-decomp} we have for any $x,y\in M$ that 
\[
Q(\Phi(x)-\Phi(y))=Q_{+}(\Phi_{+}(x)-\Phi_{+}(y))+Q_{-}(\Phi_{-}(x)-\Phi_{-}(y)).
\]
Moreover from \eqref{eq-Q-delta} and \eqref{e.delta} we know that
\[
Q(\Phi(x)-\Phi(y))=2-2\cos \delta(x,y)=4\sin^2\frac{\delta(x,y)}{2}.
\]
Let $\lambda$ be the largest absolute value of eigenvalues of $\beta_1$, then
\[
0\leqslant -Q_{-}(\Phi_{-}(x)-\Phi_{-}(y))\leqslant \lambda \cdot 4\sin^2\frac{\delta(x,y)}{2}.
\]
Hence
\[
4\sin^2\frac{\delta(x,y)}{2}\leqslant Q_{+}(\Phi_{+}(x)-\Phi_{+}(y))\leqslant 4(1+\lambda)\sin^2\frac{\delta(x,y)}{2}.
\]
This then implies that $\Phi_{+}$ is a bi-Lipschitz embedding of $(M,\delta)$ into  $(E_{+}, d_{Q_{+}})$.  Moreover, note \eqref{eq-beta1-sphere} implies that $\Phi(M)$ is a part of a hypersurface in $\R^4$, as well as $\Phi_+(M)$. Hence the dimension of $E_{+}$ is greater than the local Hausdorff dimension of $(M,\delta)$. By Proposition \ref{cor-dim-3} we obtain that
\[
\dim E_{+}> 3.
\]
Thus $3<\dim E_+\leqslant 4$. This implies $\dim E_+=4$. Combined with \eqref{eq-orth-decomp} we know that $E_+=\R^4$ and $\beta_1$ is positive definite on it.
\end{proof}

We can now define the sphere in $(\R^4,\beta_1)$ and a length distance on it. Let 
\begin{equation}\label{eq-d_Q+delta}
\cS^3:=\{u\in \R^4, Q_+(\zeta)=1\}, \quad d_{\cS^3}(u,v)=2\arcsin\left( \frac{d_ {Q_{+}}(u,v)}{2}\right),\quad u,v\in \cS^3.
\end{equation}
We are now ready to show isometry between $(M,\delta)$ and $(\mathcal{S}^3,d_{\cS^3})$ using an argument similar to \cite{CarronTewodrose2022}.
\begin{proposition}
The map $\Phi$ as given in \eqref{eq-Phi} is an isometry between $(M,\delta)$ and $(\mathcal{S}^3,d_{\cS^3})$.
\end{proposition}
\begin{proof}
First note $\Phi(x)=1$ for all $x\in M$. Combining \eqref{eq-d_Q+delta},\eqref{eq-d_Q+} and \eqref{eq-Q-delta}  we  have
\begin{equation}\label{eq-def-d-beta-1}
d_{\cS^3}(\Phi(x),\Phi(y))=\arccos (\cos r(x,y)\cos\theta(x,y))=\delta(x,y),
\end{equation}
for all $x,y\in M$. Hence $\Phi$ is an isometric embedding of $(M,\delta)$ into $(\cS^3,d_{\cS^3})$. We are left to show that $\Phi$ is onto. We shall use an argument as in \cite{CarronTewodrose2022}. If $\cS^3\setminus \Phi(M)=\varnothing$, then proof is done. Otherwise, pick a $\xi_0\in \cS^3\setminus \Phi(M)$ and let  $r_0=d_{\cS^3}(\xi_0, \Phi(M))$.

We first show that $\Phi(M)$ is contained in no hemisphere, namely $r_0<\frac{\pi}2$. If not, we have that $\Phi(M)$ is outside of the hemisphere centered at $\xi_0$.  This implies that the function $H(x):=\beta_1(\xi_0, \Phi(x))\le0$ for all $x\in M$. By the definition of $\Phi$ we know that $H$ is a linear combination of $\varphi_1(x),\dots,\varphi_{4}(x)$. Hence $H\in V_1=E_{0,1}$. Recall that $E_{0,1}$ as given in \eqref{eq-E01} is orthogonal to constant functions in $L^2(\mu)$, hence we obtain that 
\[
\int_M H(x)d\mu(x)=0
\] 
This implies that $H(x)=0$ for $\mu$-a.s.$x$ in $M$, namely $\Phi(M)\subset \{\xi\in \cS^3, d_{\cS^3}(\xi_0, \xi)=\frac{\pi}{2}\}$. It contradicts with the fact that $\dim \Phi(M)\geqslant 3$.

Next,  Proposition \ref{prop-lengthy} indicates that $(M, \delta, \mu)$ is a geodesic space, hence  the image $\Phi(M)$ is a closed and totally geodesic subset of $\cS^3$. Therefore we can find a $\xi_1\in \Phi(M)$ such that $\delta(\xi_0, \xi_1)=r_0$. On the other hand, since $\Phi$ maps geodesics to geodesics hence all minimizing paths in $\cS^3$ that joining two points in $\Phi(M)$ are contained in $\Phi(M)$. This implies that any minimizing geodesics in $\cS^3$ started from $\xi_1$ that are of length less than $\pi$ and pass through the open ball $B_{d_{\cS^3}}(\xi_0,r)$ can not touch $\Phi(M)$. However, note that the union of such minimizing geodesics form an open hemisphere which is contained in $\cS^3\setminus \Phi(M)$. This contradicts with the fact that $\Phi(M)$ is contained in no hemisphere. We then obtain the conclusion. 
\end{proof}
 
\section{The second isometry}\label{sec-isom-2}
In this section we study the second isometry for the projected space. First we introduce the quotient space and address the metric, measure and Dirichlet form on it. Then we will build the second isometry using the spectral information of its heat kernel in the later part of this section.

\subsection{The quotient space $(\B,r,\tmu)$}
Let $(M,r,\theta,\mu)$ be as given in Assumption \ref{Assump1} and let $(M,\delta,\mu)$ be as given in Corollary \ref{prop-qt-LDP}.   We can establish a fibration structure on $M$ based on the semi-metric $r:M\times M\to \R_{\geqslant 0}$. For any $x\in M$, we define the orbit (leave) that contains $x$ by
\[
\x:=\{ y\in M, r(x,y)=0\}.
\]
\begin{lemma}\label{lemma-orbit-disj}
 For any orbits $\x, \y\subset (M,\delta)$ that are not identical, they are disjoint. Moreover, each orbit is a closed set in $(M,\delta)$.
\end{lemma}
\begin{proof}
If there exists a $z\in \x\cap \y$, then we  have for any $x\in\x$ and $y\in\y$ that
\[
r(x,z)=r(y,z)=0.
\]
By the triangle inequality  we obtain that $r(x,y)=0$. This implies that $\x\subset \y$ and $\y\subset \x$, hence the first conclusion. Next, to show that each $\x$ is closed, consider a Cauchy sequence $\{x_n\}_{n\ge1}$ in $\x$, namely as $m,n\to\infty$
\[
\delta(x_n,x_m)\to 0.
\]
Combining \eqref{e.delta} with the fact that $r(x_m,x_n)=0$ we know that 
\[
\theta(x_n,x_m)\to 0.
\]
From Assumption \ref{Assump1} we know that $M$ is complete with respect to $r$ and $\theta$. Hence there exists a $x\in M$ such that $r(x_n,x)\to0$ and $\theta(x_n,x)\to0$ as $n\to\infty$. This is equivalent to say that $x\in \x$ and $\delta(x_n,x)\to0$. The second assertion is proved.  \end{proof}
The above lemma guarantees the equivalent relation on $M$ that is induced by $r$.
\begin{definition}
For any $x,y\in M$, we say $x\sim y$ if $r(x,y)=0$. We define the quotient (projective)  space by
\[
\B:=M/\sim,
\]  
which is the metric reflection of the pseudo-metric space $(M,r)$.
\end{definition}
 Note that for any $x\in \x$ and $y_1, y_2\in \y$, by triangular inequality we have that
\[
r(x,y_1)+r(y_1,y_2)\geqslant  r(x,y_2)
\]
which implies that $r(x,y_1)\geqslant  r(x,y_2)$. Interchanging $y_1$ and $y_2$ we then realize that $r(x,y_1)= r(x,y_2)$. This then allows us to extend the definition of $r$ to $\B$. Let ${r}:\B\times \B\to \R_{\geqslant 0}$ be given as
\begin{equation}\label{eq-tilde-r}
{r}(\x,\y):=r(x,y).
\end{equation}
for any $\x,\y\in \B$ and any $x\in\x$, $y\in \y$.

\begin{lemma}
The space $(\B, {r})$ is a metric space.
\end{lemma}
\begin{proof}
Clearly if any $\x,\y\in \B$ satisfy ${r}(\x,\y)=0$, then $r(x,y)=0$ for some $x\in\x$ and $y\in\y$. This implies that  $x\in\y$, hence $\x=\y$ by Lemma \ref{lemma-orbit-disj}. Symmetry is obvious. To check the triangle inequality, for any $\x,\y,\z \in \B$, 
\[
{r}(\x,\y)+{r}(\y,\z)= r(x,y)+r(y,z)\geqslant  r(x,z)={r}(\x,\z),
\]
we then obtain the conclusion.
\end{proof}
We denote by $\pi$ the projection map $\pi: M\to \B$ such that it sends any $x\in M$ to the unique orbit $\x\in\B$ that contains $x$:
\begin{align}\label{eq-pi}
\pi(x)=\x.
\end{align}
Let $\tilde{\mu}$ be the measure on $\B$ that is obtained by the pushforward of $\pi$: 
\[
\tilde{\mu}:=\mu\circ \pi^{-1}.
\]
Clearly $\tilde{\mu}(\B)=\mu(M)=1$.
Let us recall the following disintegration theorem in \cite[p.~13]{Kazukawa2022a} will be useful later.

\begin{lemma}[Theorem 3.4, \cite{Kazukawa2022a}]\label{lemma-disint}
Let $(M,\delta, \mu)$ be a metric measure space as given previously. Let $\pi: (M, \delta)\to (\B, r)$ be the projection map as given in \eqref{eq-pi}. Then there exists a family $\{\mu_\x\}_{\x\in\B}$ of probability measures on $M$ such that
\begin{itemize}
\item[(1)]For any Borel $A\subset \B$, the map $\x\mapsto \tilde{\mu}_\x(A)$ is a Borel measurable function from $\B$ to $[0,1]$.
\item[(2)] For $\tilde{\mu}$ a.s. $\x\in \B$, $\mu_\x(M\setminus \pi^{-1}(\x))=0$.
\item[(3)] For any Borel measurable function $f$ on $M$
\[
\int_M f(x)d\mu(x)=\int_\B \int_{\pi^{-1}(\x)}f(x)d\mu_\x(x)d\tilde{\mu}(\x).
\]
\end{itemize}
\end{lemma}

\subsection{Heat kernel and Dirichlet form on $(\B, {r})$} In this subsection we describe a heat kernel on the quotient space  $(\B, {r})$ given the heat kernel $q_t$ on $(M,\delta)$. From it we then obtain the associated heat semigroup and the Dirichlet form, as well as the corresponding self-adjoint operator.

\begin{definition}\label{def-tqt}
Let $q_t(x,y)$ be the heat kernel on $(M,\delta,\mu)$ as in Proposition \ref{lemma-q-t}. We define the following function $\B^2\times [0,+\infty)\to [0,1]$ by
\begin{align}\label{eq-tildePt-1}
\tq_t(\x,\y):=\int_\x \int_\y q_t(x,y)d\mu_\y (y) d\mu_\x(x),
\end{align}
where $\mu_\x$ and $\mu_\y$ are as defined in Lemma \ref{lemma-disint}.
\end{definition}
In the lemma below we prove that $\tq_t$ is in fact a heat kernel on $(\B, r, \tilde{\mu})$.
\begin{proposition}\label{lemma-tq-t}
Let $\tq_t(\x,\y)$, $t\ge0$, $\x,\y\in \B$ be as given in \eqref{eq-tildePt-1}. Then for $\tmu-$a.s. $\x, \y\in \B$ and any $s,t\ge0$,  it satisfies the conditions (i)-(v) in Definition \ref{d.HeatKernel}. Consequently  $\tq_t(\x,\y)$ is a heat kernel on $(\B, r,\tmu)$. 
\end{proposition}
\begin{proof}
Firstly (i) (ii) (iii) and (v) are easy consequences of the definition of $\tq_t$ and Proposition \ref{lemma-q-t}. We now prove (iv). We recall here the expression \eqref{eq-qt-spect} for $q_t$:
\begin{align}
q_t(x,y)=\sum_{n=0}^{\infty} \sum_{k=0}^\infty p_{k,n} (x,y)e^{-\tlam_{k,n} t}
\end{align}
where $p_{n,k}(x,y)$ are as given in \eqref{eq-p-nk} satisfying \eqref{eq-p-nk-orthog}, and
\begin{align}\label{eq-tlam}
\tlam_{k,n}=\lambda_{k,n}-n^2=-(4k(k+n+1)+2|n|+n^2).
\end{align}
For any $\x,\y\in \B$, we let
\[
\tp_{k,n}(x,\y):=\int_\y p_{k,n}(x,y)d\mu_\y(y)
\]
Also
\[
\tp_{k,n}(\x,\y):=\int_\x\int_\y p_{k,n}(x,y)d\mu_\x(x)d\mu_\y(y).
\]
Applying disintegration formula to \eqref{eq-p-nk-orthog} we then have that 
\begin{align}\label{eq-ck-equiv-pt}
\int_\B \int_\z p_{k,n} (x,z)p_{\ell,m} (z,y)d\mu_\z(z)d\tmu(\z)=\delta_{k,n}(\ell,m)p_{k,n} (x,y), \quad\text{for all }n,k,m,\ell \ge0.
\end{align}
Integrating both sides on $\x$ and $\y$ we obtain
\[
\int_\B \int_\z \tp_{k.n} (\x,z)\tp_{\ell,m} (z,\y)d\mu_\z(z)d\tmu(\z)=\delta_{k,n}(\ell,m)\tp_{k,n} (\x,\y), \quad\text{for all }n,k,m,\ell \ge0.
\]
It remains to show that 
\begin{align}\label{eq-ck-equiv}
\int_\B \int_\z \tp_{k,n} (\x,z)\tp_{\ell,m} (z,\y)d\mu_\z(z)d\tmu(\z)=\int_\B  \tp_{k,n} (\x,\z)\tp_{\ell,m} (\z,\y)d\tmu(\z).
\end{align}
We denote $F_{k,n}(x,y)=(2k+|n|+1)(\cos r(x,y))^{|n|}P_k^{0,|n|}(\cos 2r(x,y))$. By definition of $r$, we can also extend the definition of function $F_{k,n}$ to $\B\times \B$. In fact, for any $x\in\x$, $y\in\y$ we have that $F_{k,n}(x,y)=F_{k,n}(\x,\y)$. We denote
\[
\Theta_n(x,\y):= 2\int_\y \cos(n\theta(x,y))d\mu_{\y}(y) \quad \text{for } n\ge1 \text{ and } \quad \Theta_0(x,\y)=1.
\]
Additionally let
\[
\Theta_n(\x,\y):= 2\int_\x\int_\y \cos(n\theta(x,y))d\mu_{\y}(y)d\mu_{\x}(x) \quad \text{for } n\ge1 \text{ and } \quad \Theta_0(\x,\y)=1.
\]
Clearly we have 
\[
\tp_{k,n}(x,\y)=F_{k,n}(\x,\y)\Theta_n(x,\y)\quad \mbox{and}\quad \tp_{k,n}(\x,\y)=F_{k,n}(\x,\y)\Theta_n(\x,\y)
\]
Plugging the above expressions into \eqref{eq-ck-equiv} we see that it is equivalent to show
\[
 \int_\z {\Theta}_n(\x,z){\Theta}_n(\y,z)   d\mu_{\z}(z) =\Theta_n(\y,\z)\Theta_n(\x,\z),\quad \mbox{for all }x,z\in \z.
\]
It suffices to show that ${\Theta}_n(\x,z)={\Theta}_n(\x,\z)$ for all $z\in \z$, which we prove in the lemma below. The proof of present lemma is then completed.
\end{proof}

\begin{lemma}\label{lemma-Theta-delta}
Let $\Theta_n(x,\y)$ and $\Theta_n(\x,\y)$ be as given in Lemma \ref{lemma-tq-t}. Then for all $\x, \y\in \B$  and all $x\in \x$, 
\[
{\Theta}_n(x,\y)={\Theta}_n(\x,\y)=\delta_0(n).
\]
\end{lemma}
\begin{proof}

Note from \eqref{eq-ck-equiv-pt} we have that
\begin{align*}
&\int_\B \int_\z F_{k,n}(\x,\z)F_{\ell,m}(\z,\y)(e^{in\theta(x,z)}+e^{-in\theta(x,z)})(e^{im\theta(z,y)}+e^{-im\theta(z,y)})d\mu_\z(z)d\tmu(\z)\\
&\quad=\delta_{k,n}(\ell,m)p_{k,n} (x,y).
\end{align*}
Plugging in $x=y$  we have that
\begin{align}\label{eq-identity}
\delta_{k,n}(\ell,m)(|n|+1)&=\int_\B \int_\z F_{k,n}(\x,\z)F_{\ell,m}(\x,\z) \big(\cos((n-m)\theta(x,z))+\cos((n+m)\theta(x,z))\big)d\mu_\z(z)d\tmu(\z) \notag\\
&=\int_\B F_{k,n}(\x,\z)F_{\ell,m}(\x,\z)\big( \Theta_{n-m}(x,\z)+ \Theta_{n+m}(x,\z)\big)d\tmu(\z).
\end{align}
Plugging in $n=m=0$ we have that 
\[
\delta_{k}(\ell)=2\int_\B F_{k,0}(\x,\z)F_{\ell,0}(\x,\z)d\tmu(\z)
\]
for all $n\in \mathbb Z$ and $k,\ell\geqslant 0$. This implies that the families of functions $\{F_{k,0}\}_{k\ge0}$ are orthogonal in $L^2(\B,\tmu)$. We further claim that $F_{k,0}(\x,\cdot)$, $k\ge0$ form a basis in $L^2(\B,\tmu)$. 
To see this, we only need to check completeness of $F_{k,0}$ as a basis in $L^2(\B,\tmu)$.
Recall its definition 
\[
F_{k,0}(\x,\y):=(2k+1)P_k^{0,0}(\cos 2r(\x,\y)),
\]
where $P_k^{0,0}(\cdot)$ is the Laguerre polynomial. It is a known fact that 
\[
\sum_{k=0}^\infty\frac{2k+1}{2}P_k^{0,0}(a)P_k^{0,0}(b)=\delta(a-b).
\]
Let $a=\cos 2r(\x,\y)$, $b=0$ we then obtain that 
\[
\frac12\sum_{k=0}^\infty F_{k,0}(\x,\y)=\delta(\x,\y).
\]
This implies that for any $f\in L^2(\B,\tmu)$,
\[
\sum_{k=0}^\infty \int_\B F_{k,0}(\x,\y)f(\y)d\tmu(\y)=2f(\x).
\]
Therefore we have that $\int_\B F_{k,0}(\x,\y)f(\y)d\tmu(\y)=0$ for all $k\ge0$ implies that $f=0$. This proves the completeness of  $F_{k,0}(\z,\cdot )$, $k\ge0$ as an orthogonal basis in $L^2(\B,\tmu)$.

Now set $n=0$ and $\ell=0$ in \eqref{eq-identity}. Then for all $m>0, k\ge0$ and all $\x,\z \in \B$, $x\in \x$,
\begin{align}\label{eq-identity-1}
0=\int_\B F_{k,0}(\x,\z)F_{0,m}(\x,\z)\Theta_{m}(x,\z)d\tmu(\z).
\end{align}
Note that $F_{0,m}(\x,\z)=(m+1)(\cos r(\x,\z))^m$. 
Using the fact that $F_{k,0}(\x,\cdot )$, $k\ge0$ is an orthogonal basis in $L^2(\B,\tmu)$ we obtain that 
\[
(\cos r(\x,\z))^m\Theta_{m}(x,\z)\equiv0,\quad \z\in\B
\]
for all $x\in \x$, $\x\in\B$. We then obtain the conclusion.

\end{proof}

As a consequence we obtain an explicit expression for $\tq_t$.
\begin{corollary} 
Let $\tq_t(\x,\y)$ be the heat kernel on $(\B, r,\tmu)$ as given in Proposition \ref{lemma-tq-t}. Then it can be written as
\begin{align}\label{eq-tq}
\tq_t(\x,\y)=\sum_{k=0}^\infty e^{\tlam_{k,0}t}F_{k,0}(r(\x,\y)),\quad \x, \y \in\B
\end{align}
where $\tlam_{k,0}=\lambda_{k,0}=-4k(k+1)$, and $F_{k,0}(r(\x,\y))=(2k+1)P_k^{0,0}(\cos 2r(x,y))$.
\end{corollary}
\begin{proof}
The conclusion follows immediately from Lemma \ref{lemma-Theta-delta}.
\end{proof}
Using the heat kernel $\tq_t$ we can then define the associated heat semigroup $\tilde{P}_t$ on $L^2(\B, \tilde{\mu})$ by
\begin{align}\label{eq-tildeqt-1}
\tilde{P}_t f (\x)=\int_\B \tq_t(\x,\y) f(\y) d\tmu(\y), \quad t\ge0.
\end{align}

We can consequently define the Dirichlet form $\tce$ on $L^2(\tmu)$ by
\begin{equation}\label{eq-proj-E}
\tce(f,g):=\lim_{t\to0}\frac1t\int_\B {(\tilde{P}_tf-f)g}\,d\tmu, \quad f,g\in L^2(\tmu)
\end{equation}
and the self-adjoint generator $\tilde{L}$:
\begin{equation}\label{eq-proj-L}
(\tL f,g)_{L^2(\tmu)}:=\tce(f,g), \quad f,g\in L^2(\tmu).
\end{equation}

With the above result we can then obtain volume growth estimate on $(\B,r,\tmu)$ and hence the dimension estimate.
\begin{lemma} \label{lemma-Hasud-dim-B}
The Hausdorff dimension of $(\B,r,\tmu)$ is at least $2$.
\end{lemma}

\begin{proof}
From \eqref{eq-tq} know that $\tq_t(\x,\y)=\sum_{k=0}^\infty(2k+1)e^{-(4k(k+1))t}P_k^{0,0}(\cos 2r(\x,\y))$. From \cite{NowakSjogren2013} we know that for any $-1\leqslant x\leqslant 1$, $t>0$, if
\[
G_t^{\alpha,\beta}(x):=\sum_{k=0}^\infty C^{\alpha, \beta}_k\exp (-k(k+\alpha+\beta+1)t)P_k^{\alpha,\beta}(x)
\]
where $C^{\alpha, \beta}_k$ are some renormalization constants,
then there exist constant $C, c_1, c_2>0$ and $t_0>0$ such that for all $0<t<t_0$ and all $r\in[0,{\pi})$,
\[
\frac1C t^{-\alpha-1/2}(t+(\pi-r))^{-\beta-\frac12}\frac{1}{\sqrt{t}}e^{-c_1\frac{r^2}{t}}\leqslant G_t^{\alpha,\beta}(\cos r)\leqslant C t^{-\alpha-1/2}(t+(\pi-r))^{-\beta-\frac12}\frac{1}{\sqrt{t}}e^{-c_2\frac{r^2}{t}}\
\]
In our case since $\tq_t(\x,\y)=G^{0,0}_{4t}(\cos 2r(\x,\y))$, we can then easily obtain that for all $0<t<t_0$ and $\x,\y\in \B$
\[
 \frac{1}{4Ct}\frac{e^{-c_1\frac{r(\x,\y)^2}{t}}}{\sqrt{t+(\pi-2r(\x,\y))}}\leqslant \tq_t(\x,\y)\leqslant  \frac{C}{4t}\frac{e^{-c_2\frac{r(\x,\y)^2}{t}}}{\sqrt{t+(\pi-2r(\x,\y))}}.
\]
Therefore for any $0<t<t_0$ and $\x,\y\in \B$, we have
\[
 \frac{1}{4C^{\prime}t}\ e^{-c_1\frac{r(\x,\y)^2}{t}} \leqslant \tq_t(\x,\y)\leqslant  \frac{C^{\prime}}{4t\sqrt{\pi-2r(\x,\y)}}e^{-c_2\frac{r(\x,\y)^2}{t}}.
\]
At last by using classical argument  (see also \cite{CarronTewodrose2022}), we obtain that there exists $C>0$ and $\rho_0>0$  such that for any $\x\in \B$, $\rho\in(0,\rho_0)$,
\[
\tilde{\mu}(B_r(\x,\rho ))\leqslant C \rho^2
\]
where $B_r(\x,\rho )$ denotes the  ball in $(\B, r)$ centered at $\x$ of radius $\rho$. Therefore $\dim \B\geqslant 2$.
\end{proof}

\begin{proposition}\label{prop-lengthy-B}
The metric measure space $(\B,r,\tmu)$ is a geodesic space. 
\end{proposition}
\begin{proof}
First note from \eqref{eq-tq} we know that the heat kernel $\tq_t$ on $(\B,r,\tmu)$ satisfies 
\[
\frac1{Ct}e^{-\frac{r(\x,\y)^2}{4}}\le\tq_t(\x,\y)\le\frac{C}{t}e^{-\frac{r(\x,\y)^2}{4}},\quad 0<t<t_0, \ \x,\y\in\B
\]
for some $C>1$ and $t_0>0$. Combining with Lemma \ref{lemma-tq-t} (iv) and applying the same argument as in the proof of Proposition \ref{prop-lengthy}, we then obtain the conclusion.
\end{proof}

\subsection{The second isometry}
Using \eqref{eq-tq}, we apply similar argument as in the previous two sections. 
\begin{lemma}\label{lemma-spectrum-1}
Let $\left(\B, r, \tmu, \tce \right)$ as given in  and let $\tL$ be the associated self-adjoint operator as given in \eqref{eq-proj-L}. Then $\tL$ has discrete spectrum $\{\lambda'_{k,0}\}_{k\geqslant 0}$ where $\lambda'_{k,0}$ is given as in \eqref{eq-tq}. Particularly we have
\begin{equation}\label{eq-eign-fun-B}
\tL\tp_{k,0}(\x,\cdot)=\lambda'_{k,0}\tp_{k,0}(\x,\cdot)
\end{equation}
where $\tp_{k,0}(\x,\y)=F_{k,0}(r(\x,\y))$.
\end{lemma}
\begin{proof}
For any fixed $f\in L^2(\B,\tmu)$, let $(f, \tE_\lambda f)$ be the projection-valued measure of $\tL$ associated with $f$. Then from \eqref{eq-tq} we have
\[
\int_{-\infty}^\infty e^{t\lambda}d(f, E_\lambda f)=\sum_{k=0}^\infty e^{\tlam_{k,0}t}\int_{\B\times \B}\tp_{k,0}(\x,\y)f(\x)f(\y)d\tmu(\x)d\tmu(\y),
\]
By uniqueness of the map $f\to(f, \tilde{E}_\lambda f)$ we know that $\tL$ has the discrete spectrum $\{\tlam_{k,0}\}_{k\geqslant 0}$.  

We denote by $\tE_{k,0}:=\mathrm{Ker}(\tL-\tlam_{k,0}Id)$ the eigenspace in $L^2(\B, \tmu)$ that corresponds to $\tlam_{k,0}$, and by $\tP_{k,0}:L^2(M, \mu) \longrightarrow \tE_{k,0}$ the projection operator. One can then  easily verify that $\tP_{k,n}$ is an integral operator with the kernel $\tp_{k,0}$, namely for any $f\in L^2(\B, \tmu)$ we have that 
\begin{equation}\label{eq-tP-k0}
\tP_{k,0} f(\x)=\int_\B  \tp_{k,0}(\x,\y)f(\y)d\tmu (\y).
\end{equation}
Using the fact that $\tP_{k,0}$ commutes with $\tL$ and similar arguments as in \cite{CarronTewodrose2022} we obtain that $\tp_{k,0}$ is in the domain of $\tL$. Therefore \eqref{eq-eign-fun-B}
holds.
\end{proof}

Now let us focus on the eigenvalue $\tlam_{1,0}=-8$ and a corresponding eigenfunction.
 \begin{align}\label{eq-tp-10}
 \tp_{1,0}(\x,\y)=3\cos 2r(\x,\y), \quad \x,\y\in \B.
 \end{align}
For any $\x,\y\in \B$ we have that
\[
\tL_\y\cos 2r(\x,\y)=-8\cos 2r(\x,\y).
\]
\begin{lemma}
Consider  the eigenspace $\tE_{1,0}$  of $\tL$ corresponding to $\lambda'_{1,0}=-9$, then $\dim \tE_{0,1}=3$ and 
\begin{align}\label{eq-tE01}
\tE_{1,0}=\operatorname{Span}\{\cos 2r(\x,\cdot), \x\in \B\}.
\end{align}
\end{lemma}
\begin{proof}
If we denote by $\{\psi_1,\dots, \psi_\iota\}$ an orthonormal basis of the eigenspace $\tE_{1,0}$ , then by \eqref{eq-tP-k0} and \eqref{eq-tp-10}  we have for any ${f}\in L^2(\B,\tilde{\mu})$ and $\x\in \B$ that
\[
\tP_{1,0} {f}(\x)=3\int_\B \cos 2r(\x,\y){f}(\y)d\tilde{\mu}(\y)=\sum_{i=1}^\iota\left( \int_\B \psi_i(\y){f}(\y)d\tilde{\mu}(\y)\right)\psi_i(\x).
\]
Thus for any $\x,\y\in \B$,
\[
\sum_{i=1}^\iota\psi_i(\x)\psi_i(\y)=3 \cos 2r(\x,\y).
\]
Let $\y=\x$ we then obtain that
\[
\sum_{i=1}^\iota\psi_i(\x)^2=3,\quad \forall \x\in \B.
\]
Integrate both sides of the above equation we have
\[
3=\int_\B \sum_{i=1}^\iota\psi_i(\x)^2 d\tilde{\mu}(\x)=\iota.
\]
Now let $V_2:=\operatorname{Span}\{\cos (2r(\x,\cdot)), \x\in \B \}$. Clearly it is a subspace of $\tE_{1,0}$. On the other hand, for any $g\in \tE_{1,0}$, we have
\[
g(\x)=\tP_{1,0}g(\x)=3\int_\B \cos 2r(\x,\y)g(\y)d\tilde{\mu}(\y)\in V_2.
\]
Hence we have that $V_2=\tE_{1,0}$.
\end{proof}
Next consider the space $\mathcal{D}_2:=\operatorname{Span}\{\delta_{\x}(\cdot), \x\in \B\}$. Clearly it is a subspace of the algebraic dual space $V_2^*$ of $V_2$.   Moreover, note that for any ${f}\in V_2$ such that $\eta({f})=0$ for all $\eta\in \mathcal{D}_2$ we have ${f}(\x)=0$ for all $\x\in \B$. This means that
\[
\mathcal{D}_2=V_2^{\ast}.
\]
\begin{lemma}
Let $\{\x_1,\x_2, \x_3\}\in \B$ be such that $\{\delta_{\x_1},\delta_{\x_2}, \delta_{\x_3}\}$ is a basis of $V_2^*$ dual to the basis $\{\psi_1,\psi_2, \psi_3\}$ in $V_2$.  Then for $\mu-$a.s. any $\x,\y\in\B$ we have
\begin{equation}\label{eq-V-2-metric}
\cos 2r(\x,\y)=\sum_{i,j=1}^3 \cos 2r(\x_i,\x_j)\psi_i(\x) \psi_j(\y).
\end{equation}
\end{lemma}
\begin{proof}
For any $\x\in \B$, we can write
\begin{align}\label{eq-mid-cos}
\cos 2r(\x,\cdot)=\sum_{j=1}^3 \delta_{\x_j}\left(\cos 2r(\x,\cdot) \right) \psi_j(\cdot)
=\sum_{j=1}^3\cos 2r(\x,\x_j)\psi_j(\cdot).
\end{align}
Moreover, note that
\[
\cos 2r(\x,\x_j) =\cos 2r(\x_j,\x)=\sum_{i=1}^3 \cos 2r(\x_i,\x_j)\psi_i(\x),
\]
Plugging it into \eqref{eq-mid-cos} we obtain the conclusion.
\end{proof}

\begin{definition}
Clearly $\{\cos 2r(\x_i,\cdot)\}_{i=1}^3$ forms a basis of $V_2$.
We denote $b_{ij}:= \cos 2r(\x_i,\x_j)$ for all $1\leqslant i,j\leqslant 3$, it then defines a bilinear form $\beta_2$ on $\R^3\times\R^3$,  
\begin{align}\label{eq-beta2}
\beta_2 (\zeta,\zeta^{\prime})=\sum_{i,j=1}^3b_{ij}\zeta_i\zeta_j^{\prime},\quad \zeta=(\zeta_1,\dots, \zeta_3),\ \zeta^{\prime}=(\zeta^{\prime}_1,\dots, \zeta^{\prime}_3)\in\R^3.
\end{align}
Let $Q_2: \R^3\to \R$ be the associated quadratic form, i.e. for any $\zeta\in \R^3$,
\begin{align}\label{eq-Q2}
Q_2(\zeta)=\beta_2(\zeta,\zeta).
\end{align}
\end{definition}
We now define a map $\Psi: \B\to \R^3$:
\begin{align}\label{eq-Psi}
\Psi(\x):= (\psi_1(\x),\dots, \psi_3(\x)).
\end{align}
then from \eqref{eq-V-2-metric} we can write 
\[
\cos 2r(\x,\y)=\beta_2(\Psi(\x),\Psi(\y)).
\]
Particularly when $\x=\y$, we have $\beta_2(\Psi(\x),\Psi(\x))=Q_2(\Psi(\x))=1$. It implies that $\Psi(\x)\in \{\zeta\in\R^3, \beta_2(\zeta, \zeta)=1\}=: \mathcal{S}^2$. Hence $\Psi(\B)$ is a subset of $ \mathcal{S}^2$.

\begin{lemma}
Let $\Psi:\B\to\R^3$ be as given in \eqref{eq-Psi} and $Q_2$ the quadratic form on $\R^3$ as given in \eqref{eq-Q2}. Then $\Psi$ and $Q_2$ are both injective. Consequently we have $\mathrm{Ker}\, \beta_2=\{0\}$.
\end{lemma}
\begin{proof} Using the bi-linearity of $\beta_2$ we have
\begin{align}\label{eq-Q-2-delta}
Q_2(\Psi(\x)-\Psi(\y))&=\beta_2(\Psi(\x),\Psi(\x))+\beta_2(\Psi(\y),\Psi(\y))-2\beta_2(\Psi(\x),\Psi(\y))\notag
\\
&=2-2\cos 2r(\x,\y)=4\sin^2 r(\x,\y).
\end{align}
If $\Psi(\x)=\Psi(\y)$ then $Q_2=0$, thus $r(\x,\y)=0$ which implies that $\x=\y$.
\end{proof}
With the above lemma we can then orthogonally decompose $\R^3$ by
\[
\R^3=E_{+}\oplus E_{-}
\]
where $E_{+}$ and $E_{-}$ are subspaces of $\R^3$ on which $\beta_2$ is positive and negative. 
\begin{definition}
Let $\mathrm{P}_{E_+}$ and $\mathrm{P}_{E_-}$ denote the projection maps from $\R^3$ to $E_+$ and $E_-$ respectively. Define 
\begin{align}\label{eq-psi}
\Psi_{+}=:\mathrm{P}_{E_{+}}\circ \Psi,\quad \Psi_{-}=:\mathrm{P}_{E_{-}}\circ \Psi
\end{align}
Denote by $Q_2^+$ and $Q_2^-$ the restrictions of the quadratic form $Q_2$ on $E_{+}$ and $E_{-}$, i.e., for any $u\in E_{+}$ and $v\in E_{-}$,
\[
Q_2^+(u):=Q_2(P_{E_+}(u)), \quad Q_2^-(v)=Q_2(P_{E_-}(v)).
\]
We define the induced distance on $E_+$ by:
\begin{equation}\label{eq-d_Q2+}
d_{Q_2^+}(u,v):=\sqrt{Q_2^+(u-v)},\quad u,v\in E_+.
\end{equation}
\end{definition}

\begin{lemma}\label{lemma-dim1}
Let $\Psi_+$ be as given in \eqref{eq-psi}. Then $\Psi_{+}$ is a bi-Lipschitz embedding of $(\B,r)$ into $(E_{+}, d_{Q_2^{+}})$. Moreover, we have $\dim E_+=3$, consequently $E_+=\R^3$, and $\beta_2$ given in \eqref{eq-beta2} is positive definite on $\R^3$.
\end{lemma}
\begin{proof}
For any $\x,\y\in \B$, we have
\[
Q_2(\Psi(\x)-\Psi(\y))=Q_2^+(\Psi_{+}(\x)-\Psi_{+}(\y))+Q_2^-(\Psi_{-}(\x)-\Psi_{-}(\y)).
\]
Let $\lambda$ be the largest modulus of an eigenvalue of $\beta_2$, we know that
\[
0\leqslant -Q_2^-(\Psi_{-}(\x)-\Psi_{-}(\y))\leqslant \lambda \cdot 4\sin^2 r(\x,\y).
\]
Hence
\[
4\sin^2 r(\x,\y)\leqslant Q_2^+(\Psi_{+}(\x)-\Psi_{+}(\y))\leqslant 4(1+\lambda)\sin^2r(\x,\y).
\]
This then implies that $\Psi_{+}$ is a bi-Lipschitz embedding of $(\B,r)$ into a hypersurface $\cS^2$ inside $(E_{+}, d_{Q_2^+})$, hence 
\[
\dim E_+\geqslant \dim \B+1.
\]
The conclusion then follows from Lemma \ref{lemma-Hasud-dim-B}.
\end{proof}

Define the sphere in $(\R^3,\beta_2)$ and a length distance on it. Let 
\begin{equation}\label{eq-d_Q+delta2}
\cS^2:=\{u\in \R^4, Q_2^+(\zeta)=1\}, \quad d_{\cS^2}(u,v)=2\arcsin\left( \frac{d_ {Q_2^{+}}(u,v)}{2}\right),\quad u,v\in \cS^2.
\end{equation}
We are now ready to show the isometry between $(\B,r)$ and $(\mathcal{S}^2, \frac12d_{\cS^2})$ under $\Psi$, by using similar argument as previously.

\begin{proposition}
The map $\Psi$ as given in \eqref{eq-Psi} is an isometric embedding of $(\B,r)$ into $(\mathcal{S}^2,d_{\cS^2})$.
\end{proposition}
\begin{proof}
First combining \eqref{eq-Q-2-delta}, \eqref{eq-d_Q2+} and \eqref{eq-d_Q+delta2} we have that
\begin{equation}\label{eq-def-d-beta-2}
d_{\cS^2}(\Psi(\x),\Psi(\y))=2r(\x,\y),
\end{equation}
for all $\x,\y\in \B$. Hence $\Psi$ is an isometric embedding of $(\B,r)$ into $(\cS^2,\frac12d_{\cS^2})$. We are left to show that $\Psi$ is onto. We shall use similar argument as in \cite{CarronTewodrose2022}. Assume that $\Psi$ is not onto, then we can find a point $\zeta_0\in \cS^2\setminus \Psi(\B)$. Let  $2r_0=d_{\cS^2}(\zeta_0, \Psi(\B))$. 

We  claim that $\Psi(\B)$ is contained in no hemisphere.
To show the claim, it suffice to show that $2r_0<\frac{\pi}2$. Assume that $2r_0\geqslant \frac{\pi}2$, then $\Psi(\B)$ is contained outside the hemisphere centered at $\zeta_0$, which implies that $\tilde{H}(\x):=\beta_2(\zeta_0, \Psi(\x))\le0$ for all $\x\in \B$. By the definition of $\Psi$ we know that $\tilde{H}$ is a linear combination of $\psi_1(\x),\dots,\psi_{3}(\x)$. Hence $\tilde{H}\in {V}_2=\tE_{1,0}$. Recall that $\tE_{1,0}$ as in \eqref{eq-tE01} is orthogonal to constant functions in $L^2(\tmu)$, hence we obtain that 
\[
\int_\B \tilde{H}(\x)d\tmu(\x)=0.
\] 
This implies that $\tilde{H}(\x)=0$ for $\tmu$-a.s.$\x$ in $\B$, namely $\Psi(\B)\subset \{\zeta\in \cS^2, 2d_{\cS^2}(\zeta_0, \zeta)=\frac{\pi}{2}\}$. It contradicts with the fact that $\dim \Psi(\B)\geqslant 2$. Hence we obtain that $2r_0<\frac{\pi}2$.

From Proposition \ref{prop-lengthy-B} we know that $(\B, r, \tmu)$ is a geodesic space, hence  the image $\Psi(\B)$ is a closed and totally geodesic subset of $\cS^2$.
By closeness of $\Psi(\B)$ we know that $r_0>0$ and there exists a $\zeta_1$ such that $d_{\cS^2}(\zeta_0,\zeta_1)=2r_0$.

Note that  $\Psi$ maps geodesics to geodesics hence all minimizing paths in $\cS^2$ that joining two points in $\Psi(\B)$ are contained in $\Psi(\B)$.  
This implies that any minimizing geodesics in $(\cS^2,\frac12 d_{\cS^2})$ started from $\zeta_1$ that are of length less than $\frac\pi2$ and pass through the open ball $B_{\frac12d_{\cS^2}}(\zeta_0,r_0)$ can not touch $\Psi(\B)$. However, note that the union of such minimizing geodesics form an open hemisphere which is contained in $\cS^2\setminus \Psi(\B)$. This contradicts with the claim. Hence we can conclude that $\Psi$ is onto.
The proof is then completed.
 \end{proof}

\section{Riemannian submersion and bundle isometry}\label{sec-Riem-sub}  
Recall the diagram \eqref{eq-diagram}. We have already shown in previous sections the two isometries $\Phi: (M,\delta)\to (\cS^3, d_{\cS^3}) $ and $\Psi: (\B,r)\to (\cS^2,\frac12d_{\cS^2})$. Let $\pi:(M,\delta)\to (\B,r)$ be as given in \eqref{eq-pi}.

In this section we show that  $\Pi:=\Psi\circ\pi\circ \Phi^{-1}$ is a Riemannian submersion from $(\cS^3,d_{\cS^3})$ to $(\cS^2,\frac12d_{\cS^2})$ with totally geodesic fibers, which then implies the desired bundle isometry. 

The term bundle isometry was introduced in \cite{Escobales1975} in the context of connecting two equivalent Riemannian submersions. In present paper, we extend this concept for a more general situation.
\begin{definition}
Let $\Pi: M\to B$ and $\pi:N\to P$ be Riemannian submersions with totally geodesic fibers. Let $\Phi: M\to N$ and $\Psi: B\to P$ be two isometries, and the following diagram is commutative:
\[
\begin{tikzcd}
 M\arrow[r, "\Phi"] \arrow[d, "\pi"]
    &  N \arrow[d, "\Pi"] \\
  B \arrow[r, "\Psi"]
&  P 
\end{tikzcd}
\]
Then the pair $(\Phi,\Psi)$ is called bundle isometry between $\Pi$ and $\pi$.
\end{definition}
\begin{theorem}
Let $\Pi=\Psi\circ\pi\circ \Phi^{-1}$. It is a  Riemannian submersion from $(\cS^3,d_{\cS^3})$ to $(\cS^2,\frac12d_{\cS^2})$. Consequently $(\Phi,\Psi)$ is a bundle isometry between $\Pi$ and $\pi$.
\end{theorem}
\begin{proof}
Step 1: 

First we induce a fibration structure ${\Gamma}$ on $(\cS^3,d_{\cS^3})$ by using the isometry $\Phi$. We define equivalent classes on $\cS^3$ setting $u\sim v$ if $\Phi^{-1}(u)\sim \Phi(v)^{-1}$ in $M$, namely $r(\Phi^{-1}(u), \Phi^{-1}(v))=0$. 

Since 
\begin{align*}
r(\Phi^{-1}(u), \Phi^{-1}(v))&={r}(\pi\circ\Phi^{-1}(u), \pi\circ\Phi^{-1}(v))\\
&=\frac12d_{\cS^2}(\Psi\circ\pi\circ\Phi^{-1}(u),\Psi\circ \pi\circ\Phi^{-1}(v))=\frac12d_{\cS^2}(\Pi(u), \Pi(v))
\end{align*}
the equivalent relation on $\cS^3$ therefore coincides with $u\sim v$ iff $d_{\cS^2}(\Pi(u), \Pi(v))=0$.

Step 2: We show that the fibers are $1$-dimensional totally geodesic fibers.

Recall the definition of \emph{totally geodesic}: a submanifold $N$ of a Riemannian manifold $M$ is totally geodesic if any geodesic on $N$ with the induced Riemannian metric is a geodesic on $M$. Clearly any geodesic on $(\cS^3,d_{\cS^3})$ can be viewed as a totally geodesic submanifold of dimension $1$. Let $\mathbf{u}$ be a fiber in $\cS^3$ obtained as in step  1. It is enough to show that for any $u,v\in \mathbf{u}$, the (arc-length parametrized) geodesic path $\gamma[0,d_{\cS^3}(u,v)]$ connecting $u$ and $v$ is in fact in $\mathbf{u}$. We prove it in the lemma below.

\begin{lemma}
For any $u,v$ in the same fiber $\mathbf{u}\subset \cS^3$, let $\gamma[0, d_{\cS^3}(u,v)]$ be the arc-length parametrized geodesic path that connects $u$ and $v$, i.e. $\gamma(0)=u$, $\gamma(d_{\cS^3}(u,v))=v$ and $|\dot{\gamma}(t)|=1$ for all $t$. Then we have that $\gamma\subset \mathbf{u}$. 
\end{lemma}
\begin{proof}
By the definition of a geodesic 
\[
d_{\cS^3}(u,v)=\sup\sum_{i=0}^{n-1}d_{\cS^{3}}(\gamma(s_i),\gamma(s_{i+1}))
\]
where the supremum is taken over all $n\in \mathbb{Z}_+$ and all partitions $0= s_0\leqslant s_1\leqslant \cdots \leqslant s_n=d_{\cS^3}(u,v)$.  By the definition of $\delta$ (see \eqref{eq-def-d-beta-1}) we have for each $i$ that
\begin{align*}
d_{\cS^3}(\gamma(s_i),\gamma(s_{i+1}))&=\delta(\Phi^{-1}(\gamma(s_i)), \Phi^{-1}(\gamma(s_{i+1})))\\
&=\arccos\bigg(\cos r\big(\Phi^{-1}(\gamma(s_i)),\Phi^{-1}\gamma(s_{i+1}))\big)\cos\theta\big(\Phi^{-1}(\gamma(s_i)),\Phi^{-1}(\gamma(s_{i+1}))\big)\bigg)\\
&\geqslant \theta\big(\Phi^{-1}(\gamma(s_i)),\Phi^{-1}(\gamma(s_{i+1}))\big)
\end{align*}
Meanwhile for $u\sim v\in \cS^3$,
\[
d_{\cS^3}(u,v)=\delta(\Phi^{-1}(u), \Phi^{-1}(v))=\theta(\Phi^{-1}(u), \Phi^{-1}(v))
\]
Hence we have that
\[
\theta(\Phi^{-1}(u), \Phi^{-1}(v))\geqslant \sum_{i=0}^{n-1} \theta\big(\Phi^{-1}(\gamma(s_i)),\Phi^{-1}(\gamma(s_{i+1}))\big)\geqslant \theta(\Phi^{-1}(u), \Phi^{-1}(v))
\]
The last inequality follows from the triangle inequality that is satisfied by the semi-metric $\theta$. Hence we obtain that
\[
r\big(\Phi^{-1}(\gamma(s_i)),\Phi^{-1}(\gamma(s_{i+1}))\big)=0
\]
for all $1\leqslant i\leqslant n$ in any partition $0= s_0\leqslant s_1\leqslant \cdots \leqslant s_n=d_{\cS^3}(u,v)$, which implies that $\gamma\in \mathbf{u}$. 
\end{proof}

Step 3: At last we show that $\Pi: (\cS^3,d_{\cS^3})\to (\cS^2, \frac12d_{\cS^2})$ is a Riemannian submersion. Recall that the definition of Riemannian submersion $\Pi:X\to Y$ is such that $d\Pi: (\ker(d\Pi))^{\perp}\to T(Y)$ is an isometry. We call a path $\gamma(t), t\in[0,1]$ \textbf{horizontal} if  $\dot{\gamma}(t)\perp Z_{\gamma(t)}$ for a.e. $t\in[0,1]$, where $Z$ denotes the fiber direction. 
It suffices to show that the length of a horizontal path is preserved by $\Pi$, i.e., for any $t\in[0,1]$,
\[
\ell_{\cS^3}( \gamma[0,t]))=\frac1{2}\ell_{\cS^2}(\Pi\circ\gamma[0,t]).
\]
Recall that by definition we know that
\[
\ell_{\cS^3}(\gamma[0,t])=\sup\sum_{i=0}^{n-1}d_{\cS^3}(\gamma(s_i),\gamma(s_{i+1}))
\]
where the supremum is taken over all $n\in \mathbb{Z}_+$ and all partitions $0= s_0\leqslant s_1\leqslant \cdots \leqslant s_n=1$. It suffices to show that
\[
d_{\cS^3}(\gamma(s_i), \gamma(s_{i+1}))=\frac1{2}d_{\cS^2}(\Pi\circ\gamma(s_i),\Pi\circ  \gamma(s_{i+1})).
\]
for all $s_i, s_{i+1}$ is the same sufficiently small neighborhood. From \eqref{eq-def-d-beta-1} and \eqref{eq-def-d-beta-2} we know that it suffices to show that
\begin{align*}
&\arccos (\cos r \left(\Phi^{-1}(\gamma(s)),\Phi^{-1}(\gamma(s')) \right)\cos\theta\left(\Phi^{-1}(\gamma(s)),\Phi^{-1}(\gamma(s'))\right)\\
&=r(\Psi^{-1}\circ\Pi(\gamma(s)),\Psi^{-1}\circ\Pi(\gamma(s')))
\end{align*}
for any $s,s'\in[0,1]$ in a small enough neighborhood.  Note that 
\begin{align*}
r(\Psi^{-1}\circ\Pi(\gamma(s)),\Psi^{-1}\circ\Pi(\gamma(s')))&=r(\pi\circ\Phi^{-1}(\gamma(s)),\pi\circ\Phi^{-1}(\gamma(s')))\\
&=r(\Phi^{-1}(\gamma(s)),\Phi^{-1}(\gamma(s')))
\end{align*}
therefore it suffices to show that 
\begin{align}\label{eq-th-ga}
\theta(\Phi^{-1}(\gamma(s)),\Phi^{-1}(\gamma(s')))=0
\end{align}
for any horizontal path $\gamma$ and any $s,s'\in[0,1]$ in a small enough neighborhood.

We denote by $\tth(\cdot,\cdot):=\theta(\Phi^{-1}(\cdot), \Phi^{-1}(\cdot))$ the push-ward of $\theta$ on $\cS^3$. It can be easily verified that $\tth$ is a semi-metric on $\cS^3$. In particular, when restricted on fibers, $\tth$ coincides with $d_{\cS^3}$. To see this, let $\mathbf{u}\subset \cS^3$ be the fiber that contains $u$, and pick a point $v\in \mathbf{u}$ that is a neighborhood of $u$. Let $\sigma[0,1]$ be a geodesic path that connects $u$ and $v$, i.e. $\sigma(0)=u$, $\sigma(1)=v$. From \eqref{eq-def-d-beta-1} we know that for all $0\leqslant s\leqslant t\le1$,
\begin{align}\label{eq-ds-tth}
d_{\cS^3}(\sigma(s), \sigma(t))=\theta(\Phi^{-1}(\sigma(s)), \Phi^{-1}(\sigma(t))= \tth(\sigma(s),\sigma(t)).
\end{align}
To prove \eqref{eq-th-ga}, we introduce the notation $\tth_{u}(\cdot):= \tth(u,\cdot)$. Then since
\[
\tth(\gamma(s),\gamma(s'))=\int_s^{s'} \langle\nabla \tth_{\gamma(t)}(\gamma(t)),\dot{\gamma}(t)\rangle_{\beta_1}dt
\]
 it suffices to show that for all $s\leqslant t \leqslant s'$,
\[
\nabla\tth_{\gamma(t)}(\gamma(t))\perp \dot{\gamma}(t).
\]
Taking into the account that $\gamma$ is a horizontal path, which implies that $ \dot{\gamma}(t)\perp Z_{\gamma(t)}$. Therefore it suffices to show that
\begin{equation}\label{eq-desired-parelle}
\nabla\tth_{u}(u) \ ||\ Z_{u},\quad \mbox{for all }u\in \cS^{3}.
\end{equation}
To show this, let $\mathbf{u}\subset \cS^3$ be the fiber that contains $u$, and pick a point $v\in \mathbf{u}$ that is a neighborhood of $u$. Let $\sigma[0,1]$ be a geodesic path that connects $u$ and $v$, i.e. $\sigma(0)=u$, $\sigma(1)=v$. From \eqref{eq-ds-tth} we know that
\[
\int_0^t|\dot{\sigma}(s)|_{\beta_1}ds= \int_0^t \langle\nabla \tth_{\sigma(s)}(\sigma(s)),\dot{\sigma}(s)\rangle_{\beta_1}ds.
\]
Taking derivative at $t=0$ on both sides of the above expression we then obtain that
\begin{align*}
|\dot{\sigma}(0)|_{\beta_1}=\langle\nabla\tth_{\sigma(0)}(\sigma(0)), \dot{\sigma}(0)\rangle_{\beta_1},
\end{align*}
which implies that $\dot{\sigma}(0)\ ||\ \nabla\tth_{\sigma(0)}(\sigma(0))$, which is  \eqref{eq-desired-parelle}. 

Lastly we can easily observe that $M\to \B$ is a submersion with totally geodesic fibers by using similar argument as in Step 2. We then obtain the desired conclusion. 
\end{proof}

\begin{acknowledgement}
The authors thank David Tewodrose and Gilles Carron for the discussion concerning the spectrum of the operator $L$ including the argument used to prove Proposition~\ref{p.HeatKernel}. This can be applied to the spherical case in \cite{CarronTewodrose2022} as well. 
\end{acknowledgement}

\bibliographystyle{amsplain}

\providecommand{\bysame}{\leavevmode\hbox to3em{\hrulefill}\thinspace}
\providecommand{\MR}{\relax\ifhmode\unskip\space\fi MR }
\providecommand{\MRhref}[2]{%
  \href{http://www.ams.org/mathscinet-getitem?mr=#1}{#2}
}
\providecommand{\href}[2]{#2}

\end{document}